\documentclass{article}
\usepackage[utf8]{inputenc}

\usepackage{amssymb}
\usepackage{amsthm}
\usepackage{amsmath,amscd}
\usepackage[mathscr]{euscript}
\usepackage[all]{xy}
\usepackage{lmodern}
\usepackage[T1]{fontenc}
\usepackage[textwidth=14cm,hcentering]{geometry}
\usepackage[colorlinks=true,linkcolor=red,citecolor=blue]{hyperref}
\usepackage{cleveref}
\usepackage{mathtools}
\usepackage{xspace}
\usepackage{bbm}
\usepackage{tikz}
\usetikzlibrary{cd}
\usetikzlibrary{arrows,positioning}
\usepackage{caption}
\usepackage{aliascnt}
\title{Free rigid commutative algebras}
\author{Maxime Ramzi}
\date{}

\newtheorem{thm}{Theorem}[section]
\newaliascnt{lm}{thm}  
\newtheorem{lm}[lm]{Lemma}
\aliascntresetthe{lm}
\Crefname{lm}{Lemma}{Lemmas}
\newaliascnt{prop}{thm}  
\newtheorem{prop}[prop]{Proposition}
\aliascntresetthe{prop}
\Crefname{prop}{Proposition}{Propositions}
\newaliascnt{cor}{thm}  
\newtheorem{cor}[cor]{Corollary}
\aliascntresetthe{cor}
\Crefname{cor}{Corollary}{Corollaries}

\newtheorem*{thm*}{Theorem}
\newtheorem*{cor*}{Corollary}

\theoremstyle{definition}
\newaliascnt{defn}{thm}  
\newtheorem{defn}[defn]{Definition}
\aliascntresetthe{defn}
\Crefname{defn}{Definition}{Definitions}
\newaliascnt{cons}{thm}  
\newtheorem{cons}[cons]{Construction}
\aliascntresetthe{cons}
\Crefname{cons}{Construction}{Constructions}
\newaliascnt{nota}{thm}  
\newtheorem{nota}[nota]{Notation}
\aliascntresetthe{nota}
\Crefname{nota}{Notation}{Notations}
\newaliascnt{conv}{thm}  

\aliascntresetthe{conv}
\Crefname{conv}{Convention}{Conventions}
\newaliascnt{ex}{thm}  
\newtheorem{ex}[ex]{Example}
\aliascntresetthe{ex}
\Crefname{ex}{Example}{Examples}
\newaliascnt{rmk}{thm}  
\newtheorem{rmk}[rmk]{Remark}
\aliascntresetthe{rmk}
\Crefname{rmk}{Remark}{Remarks}
\newaliascnt{ques}{thm}  
\newtheorem{ques}[ques]{Question}
\aliascntresetthe{ques}
\Crefname{ques}{Question}{Questions}
\newaliascnt{conj}{thm}  

\aliascntresetthe{conj}
\Crefname{conj}{Conjecture}{Conjectures}
\newaliascnt{warn}{thm}  
\newtheorem{warn}[warn]{Warning}
\aliascntresetthe{warn}
\Crefname{warn}{Warning}{Warnings}
\newaliascnt{obs}{thm}  
\newtheorem{obs}[obs]{Observation}
\aliascntresetthe{obs}
\Crefname{obs}{Observation}{Observations}
\newtheorem*{ques*}{Question}
\newtheorem*{rmk*}{Remark}
\newtheorem*{ex*}{Example}
\newaliascnt{recoll}{thm}  
\newtheorem{recoll}[recoll]{Recollection}
\aliascntresetthe{recoll}
\Crefname{recoll}{Recollection}{Recollections}

\newcommand{\op}{^{\mathrm{op}}}
\newcommand{\cat}{\mathrm}
\newcommand{\Cat}{\cat{Cat}}
\newcommand{\on}{\operatorname}
\newcommand{\id}{\mathrm{id}}
\newcommand{\Fun}{\on{Fun}}

\newcommand{\Map}{\on{Map}}
\newcommand{\Hom}{\on{Hom}}

\newcommand{\Frrig}{\mathrm{Fr}^\rig}
\newcommand{\ev}{\mathrm{ev}}

\newcommand{\Dim}{\on{Dim}}

\newcommand{\cst}{\mathrm{cst}}
\newcommand{\V}{\mathcal{V}}
\newcommand{\M}{\mathcal{M}}
\newcommand{\N}{\mathbb N}
\newcommand{\Z}{\mathbb Z}

\newcommand{\st}{\mathrm{st}}

\newcommand{\Ab}{\cat{Ab}}

\newcommand{\An}{\cat{An}}
\newcommand{\Sp}{\cat{Sp}}

\newcommand{\cart}{\cat{Cart}}
\newcommand{\PrL}{\cat{Pr}^\mathrm{L} }
\newcommand{\rex}{\mathrm{rex}}

\newcommand{\lax}{\mathrm{lax}}

\newcommand{\CAlg}{\mathrm{CAlg}}

\newcommand{\Mod}{\cat{Mod}}

\newcommand{\Cob}{\mathrm{Cob}}
\newcommand{\HH}{\mathrm{HH}}
\newcommand{\THH}{\mathrm{THH}}

\newcommand{\Prof}{\cat{Prof}}

\newcommand{\Ind}{\mathrm{Ind}}
\newcommand{\End}{\mathrm{End}}

\newcommand{\pt}{\mathrm{pt}}
\newcommand{\colim}{\mathrm{colim}}
\newcommand{\Lax}{\mathrm{Lax}}

\newcommand{\Psh}{\cat{Psh}}

\newcommand{\dbl}{\mathrm{dbl}}
\newcommand{\rig}{\mathrm{rig}}

\newcommand{\oplax}{\mathrm{oplax}}

\newcommand{\one}{\mathbbm{1}}

\newcommand{\EE}{\mathbb{E}}
\newcommand{\B}{\mathbf{B}}

\newcommand{\coCart}{\mathrm{coCart}}
\newcommand{\Ar}{\mathrm{Ar}}

\newcommand{\W}{\mathcal W}

\DeclareFontFamily{U}{min}{}
\DeclareFontShape{U}{min}{m}{n}{<-> udmj30}{}

\newcommand{\bigboxtimes}{%
  \mathop{\vcenter{\hbox{\Large$\boxtimes$}}}%
}

\newcommand{\category}{\xspace{$(\infty,1)$-category}\xspace}
\newcommand{\categories}{\xspace{$(\infty,1)$-categories}\xspace}
\newcommand{\twocategory}{\xspace{$(\infty,2)$-category}\xspace}
\newcommand{\twocategories}{\xspace{$(\infty,2)$-categories}\xspace}

\begin{document}

\maketitle
\begin{abstract}
    We describe free rigid commutative algebras in $2$-presentably symmetric monoidal $(\infty,2)$-categories as oplax colimits over the $1$-dimensional framed cobordism category. The special case of the $(\infty,2)$-category $\PrL$ itself provides a description of the free symmetric monoidal $(\infty,1)$-category with duals on a given $(\infty,1)$-category, while the case of $\Mod_\V(\PrL)$ provides a description of a similar object in the $\V$-enriched context, for $\mathcal{V}$ a presentably symmetric monoidal $(\infty,1)$-category.

    As a byproduct, we obtain new proofs of some results about rigidification of locally rigid categories, as well as a proof that any rigid category over $\Sp$ embeds into a compactly-rigidly generated one.  
\end{abstract}
\begin{center}

\includegraphics[scale=0.20]{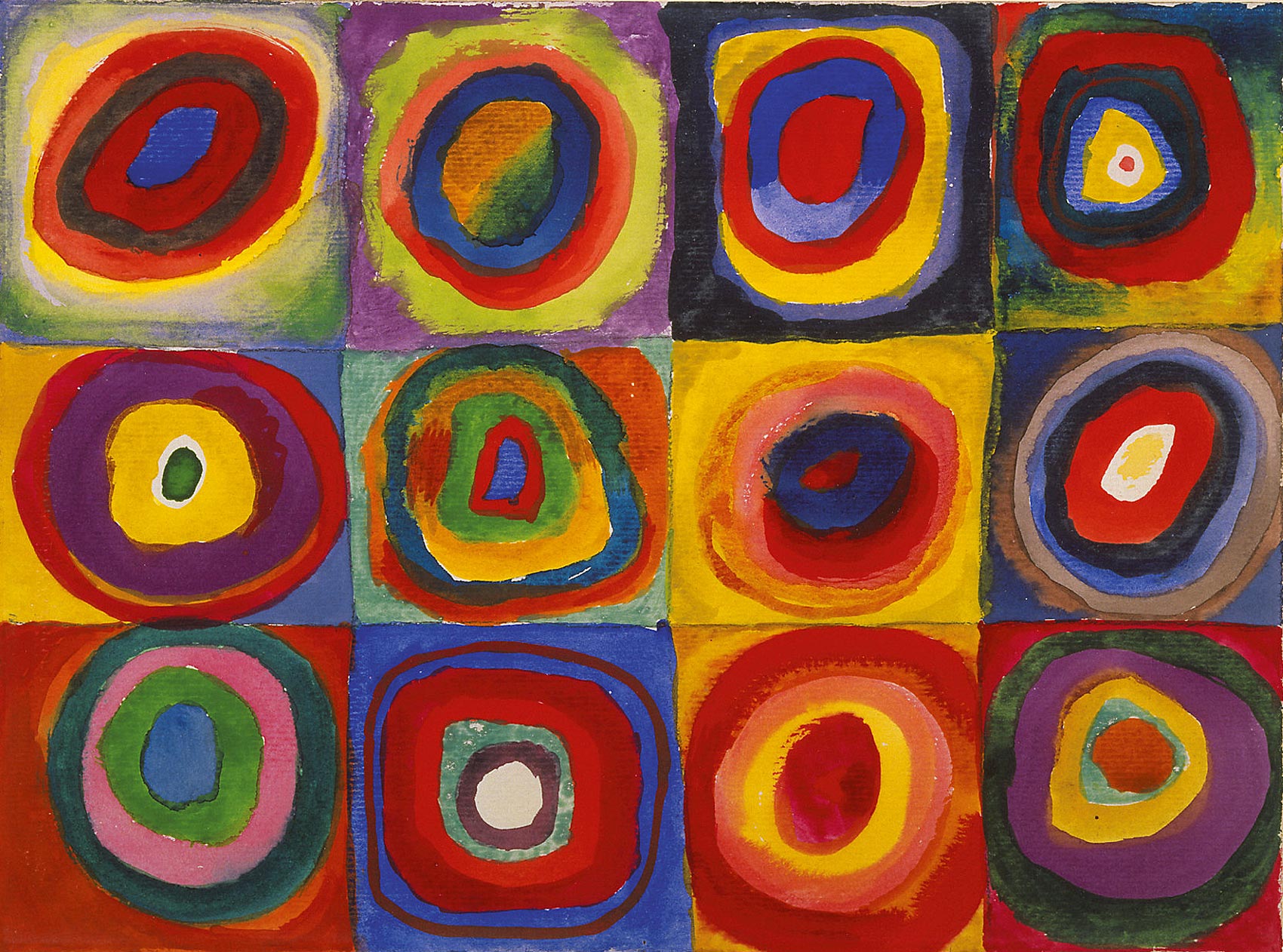}
\captionof{figure}{Farbstudie - Quadrate und konzentrische Ringe, W. Kandinsky}
\end{center}
    
\newpage
\tableofcontents
\section*{Introduction}
Symmetric monoidal categories are some kind of categorification of commutative rings, and among all (additive, say) symmetric monoidal categories, the ones that behave the most ``like'' commutative rings are the \emph{rigid} ones, that is\footnote{Though we will see that the meaning of this word has expanded in recent years.} the ones where every object is dualizable. These feature in many areas of mathematics, like tensor-triangular geometry or the theory of motives and are of central importance. Like for any algebraic or categorical gadget, it is therefore a reasonable to ask how to describe \emph{free} rigid symmetric monoidal categories. The goal of the present article is to give a somewhat general description of these objects. 

For example, in the context of additive categories, the description of the free rigid category on a single object goes back at least to work of Deligne \cite[§10]{deligne} and Deligne--Milne \cite[Example 1.27]{delignemilne} and their famous $\mathrm{Rep}(GL_t)$ categories. These were used in recent work of Barthel, Hyslop and the author to disprove Balmer's Nerves of Steel conjecture \cite{BHR}, showing their general applicability and the importance of our ability to describe them.

The question of such a description also makes sense for $(\infty,1)$-categories,  where the Cobordism Hypothesis, conjectured by Baez--Dolan \cite{baezdolan}, describes the free rigid symmetric monoidal $(\infty,1)$-category in terms of $1$-dimensional cobordisms. More generally, the cobordism hypothesis describes the free symmetric monoidal $(\infty,n)$-category on a fully dualizable object in terms of higher dimensional cobordisms.  

Lurie sketched a proof of the Cobordism Hypothesis in all dimensions in \cite{luriecob}, and a complete proof in dimension $1$ was given by Harpaz \cite{harpazcob}. In fact, even in dimension $1$, there are various generalizations of the cobordism hypothesis one can conjecture, describing more general rigid categories defined by generators and relations. In \cite{Jan}, Barkan and Steinebrunner reprove the $1$D cobordism hypothesis, but their methods are general enough to describe symmetric monoidal $(\infty,1)$-categories with duals free on any $(\infty,1)$-category\footnote{In fact, they do much more, but we will not discuss the further generalizations they provide in this article.}.

They also do this in an enriched setting: given a presentably symmetric monoidal $(\infty,1)$-category $\V$, they describe free symmetric monoidal $\V$-enriched $(\infty,1)$-categories on a given $\V$-enriched $(\infty,1)$-category. In the context of $\V=\Ab$, this recovers the descriptions given by Deligne and Milne mentioned earlier. 

The shape of the results they kindly shared with the author suggests an alternative description of these free objects, and the goal of this article is to prove this alternative description with a different method, which in turn generalizes it in a different direction.

To explain our perspective and the relevant generalization (as well as the proof), we start by recalling that the functor $C\mapsto \Psh(C)$, assigning to an \category its presentable \category of presheaves, while not fully faithful, is not far from so: we simply need to neglect idempotent-completion, and to restrict our attention to colimit-preserving functors with a colimit-preserving right adjoint. In this sense, we can think of \categories as a special case of presentable \categories, and symmetric monoidal \categories as presentably symmetric monoidal \categories. With a base presentably symmetric monoidal $\V$, a similar story works, replacing presentable \categories with presentable $\V$-modules, and symmetric monoidal $\V$-enriched \categories with commutative $\V$-algebras in $\PrL$. The advantage of this point of view, developped to a large extent in \cite{reutterzetto}, is that it remains wholly in the language of the higher algebra of \categories, and is thereby easier to manipulate. 

From this perspective, the notion that plays the role of a symmetric monoidal ($\V$-enriched) \category \emph{with duals} is that of a \emph{rigid} commutative $\V$-algebra, as introduced by Gaitsgory and Rozenblyum \cite[Chapter 1, §9]{gaitsroz} and studied by the author in \cite{LocRig}. We refer to \Cref{section:rig} for a precise definition, but we note already that unlike ``every object in $C$ has a dual'', the statement ``$\Psh(C)$ is rigid'' can be stated completely in the language of symmetric monoidal \twocategories, interpreted in the symmetric monoidal \twocategory $\PrL$. It therefore makes sense to ask, given a symmetric monoidal \twocategory $\B$, whether there exist free rigid commutative algebras, and if yes, how one can describe them. It turns out that a reasonable generality in which they might exist is for \emph{dualizable} objects in $\B$, which play the role of presentable \categories of the form $\Psh(C)$, but strictly generalize them. 
\begin{ex*}
    In the context of $\B=\PrL_\st$, the author showed in \cite{Dbl} that a certain category $\Pr^\dbl_\st$ of dualizable presentable stable \categories is itself presentable, and in \cite{LocRig} that the forgetful functor from rigid commutative algebras to $\Pr^\dbl_\st$ preserves limits. Hence, by the adjoint functor theorem, it admits a left adjoint: free rigid commutative algebras on dualizable objects exist. 
\end{ex*}

In the description obtained by Barkan and Steinebrunner in \cite{Jan}, the mapping anima in the free symmetric monoidal \category with duals on a given \category $C$ involve colimits of mapping anima in $C$, and therefore it seems reasonable to expect that one needs $\B$ to be sufficiently \emph{locally} cocomplete to make sense of this, that is, that its hom categories are sufficiently cocomplete. We define in the main text a notion of weakly $2$-presentably symmetric monoidal \twocategories (see \Cref{defn:weak2prl})
that encodes this. With this in mind, an informal version of our main result is as follows:
\begin{thm*}[\Cref{thm:maintext}]
    Let $\B$ be a weakly $2$-presentably symmetric monoidal \twocategory, and $b\in \B$ be a dualizable object. There exists a free rigid commutative algebra on $b$ in $\B$, and it is given by a canonical commutative algebra structure on $$\colim^\oplax_{M^+\coprod N^-\in\Cob^{\mathrm{1d,or}}} b^{\otimes M}\otimes (b^\vee)^{\otimes N} ,$$ the oplax colimit of the unique symmetric monoidal functor $\Cob^{\mathrm{1d,or}}\to \B$ sending the positively oriented point $+$ to $b$. 
\end{thm*}
\begin{ex*}
   The symmetric monoidal \twocategory $\B=\Mod_\V(\PrL)$ is weakly $2$-presentably symmetric monoidal, and $\V$-presheaf \categories of small $\V$-enriched \categories are dualizable, so that this result applies to our motivating example.  
\end{ex*}

Before going further, let us make two comments about the generality in which this theorem is stated and proved. First, we note that even in the case where $\B=\Mod_\V(\PrL)$, our motivating example, there are strictly more dualizable objects than there are (Cauchy complete) small $\V$-enriched \categories. For example, for $\V=\Sp$, dualizable presentable stable \categories have been the subject of growing attention after Efimov's work on his \emph{continuous $K$-theory}, and there, the examples are manifold:
\begin{ex*}
    Let $X$ be a locally compact Hausdorff topological space. The \category of sheaves of spectra on $X$, $\mathrm{Sh}(X,\Sp)$, is dualizable in $\PrL_\st$, but rarely compactly generated, cf. \cite{harr}, \cite[Section 6.4]{efimovI}.
\end{ex*}
Thus, unlike Barkan and Steinebrunner's work, which is more general in some other directions, our result applies also to dualizable objects in $\Mod_\V(\PrL)$ which do not arise as $\V$-presheaf \categories. We note that surprisingly, say for $\PrL$ itself, our proof is no more difficult in the generality of dualizable objects than in the generality of objects of the form $\Psh(C)$. Thus, this added generality is ``for free''.  

Second, since we have stressed so heavily that we mostly cared about $\B=\Mod_\V(\PrL)$, the reader may wonder why the result is stated in the given generality. Besides the fact that this is the natural generality in which the question of free rigid algebras can be asked, let us note that the generality is also heavily \emph{used} in the proof. Indeed, while there is no universal $\V$ equipped with a dualizable module, there \emph{is} a universal $\B$ equipped with a dualizable object\footnote{At this point, the reader may already have guessed what this $\B$ is and they might start to get a feeling for why the answer has to look like it does.}. This allows us to contemplate the free rigid commutative algebra on this universal dualizable object, and ask how to describe it. What one finds, in that case, is that the situation is somewhat degenerate and therefore easy to analyze, essentially because of this universality. This is ultimately how our proof works. 

Now, Barkan and Steinebrunner obtain rather explicit formulae for the mapping anima in the free symmetric monoidal \category with duals on a given \category $C$, and one may worry that \Cref{thm:maintext} does not quite provide that. We also show that the situation is actually rather nice, and we do obtain formulae -- these were in fact what originally motivated the author, given the applications of the work of Barkan and Steinebrunner to the author's \cite{KNS}. We give more detailed formulae for other mapping objects in the body of the text, but for now let us extract one of the key results we were after: 
\begin{cor*}[{\Cref{cor:End1}}]
    Let $\B$ be a weakly $2$-presentably symmetric monoidal \twocategory with unit $\one_\B$, and let $b\in \B$ be dualizable. Let $\Frrig_\B(b)$ denote the free rigid commutative algebra on $b$. 

    In the symmetric monoidal \category $\End(\one_\B)$, the endomorphism ring object of $\one_{\Frrig_\B(b)}$ is the free commutative algebra on $\Dim(b)_{hS^1}$, where $\Dim(b)$ is the symmetric monoidal dimension of $b$ with its $S^1$-action. 
\end{cor*}
\begin{ex*}
    When $\V=\Ab$ and $b=\Fun^\times(C\op, \Ab)$ for some small additive category $C$, our main theorem gives a description of $\Fun^\times(F(C)\op,\Ab)$ where $F(C)$ is the free additively symmetric monoidal category with duals on $C$. For example when $C$ is the category of finitely generated free modules over $\Z$, this gives a description of the free additively symmetric monoidal category on a dualizable object. This was already described rather elementarily, for example in \cite[Example 1.26]{delignemilne} (see also \cite[Définition 10.2, Proposition 10.3]{deligne}) -- in this case one is trying to describe a $1$-category, and so one can establish a description ``by hand'', and this what these references do. 
    
    In particular, the $S^1$-orbits from the above corollary are invisible in $\Ab$. The more detailed formula from \Cref{cor:maintextV} shows how to recover the full description given in \emph{loc. cit.} in this example.
\end{ex*}
The fact that our formula is so simple turns out to have surprising byproducts in the theory of rigid commutative algebras that the author did not really expect. Specifically, we give a new and simpler proof of \cite[Theorem 4.67]{LocRig}, as well as generalize (and give what looks like a simpler proof of) \cite[Proposition 4.1]{efimovII}. For details, we invite the interested reader to consult \Cref{section:locrig}. However, let us also point out the following surprising corollary of our formula, which answers a question posed to us by Nikolaus (to which we expected a negative answer):
\begin{cor*}[\Cref{cor:ffrigidemb}]
    Let $C\in \CAlg(\PrL_\st)$ be a rigid commutative algebra. There exists a rigid, \emph{compactly generated} commutative algebra $D$ and a fully faithful embedding $C\to D$ in $\CAlg(\PrL_\st)$.
 \end{cor*}

\begin{rmk}\label{rmk:En}
    We restrict throughout to symmetric monoidal \twocategories and their rigid commutative algebras, but for every $n\geq 1$, there is a reasonable notion\footnote{This is also the case for $n=0$, but there the story is rather trivial.} of rigid $\EE_n$-algebra in an $\EE_n$-monoidal \twocategory. One can ask similar questions and try to get a similar answer, depending this time, not on $\Cob^{\mathrm{1d,or}}$, the free symmetric monoidal \category on a dualizable object, but rather on the free $\EE_n$-monoidal \category on a dualizable object, describable also in cobordism-theoretic terms \cite{tangle}. A lot of our methods work, but as far as we can tell, not all of them directly generalize, and in particular at this point the author does not know how to make the proof go through in that generality. It would certainly be an interesting question to investigate.  

    For the interested reader, the particular point I cannot replicate in the $\EE_n$-story is \Cref{cons:cocone}. One possible approach would be to generalize the work of Neuhauser \cite{neuhauser} in the $\EE_n$-monoidal case and see whether in the universal case there is such an oplax cocone. 
\end{rmk}
\begin{rmk*}
    Another potentially interesting venue for generalization, which we do not adress, lies in the ``higher'' direction. One can hope that the higher cobordism categories appear in the theory of free $n$-rigid commutative algebras in symmetric monoidal $(\infty,n)$-categories -- these gadgets are not yet defined in the literature, but it appears Aoki has given a definition in his thesis. It is also natural to wonder about cobordism categories with tangential structures, and what they encode on the rigid side of things. 
\end{rmk*}
\subsection*{Relation to other work}
The relation to Barkan and Steinebrunner's work \cite{Jan} is hopefully evident. Their results predate ours, and in fact our proof originated as an attempt to reverse-engineer the answer from the description they obtained and generously shared with the author. In particular, their results already cover the case that motivated the author. Thus, one should think of the first part of the present article as a new and perhaps interesting proof of their results.
It is also worth noting that while our results are slightly more general in the sense described in the discussion following the statement of the main theorem, their results are also more general in a different direction, namely they do \emph{not} restrict to describing free symmetric monoidal \categories with duals, and rather also describe other kinds of free constructions, e.g. the free symmetric monoidal category on a dualizable object $x$ equipped with an arrow $\one\to x$, which features in \cite{BHR}. These kinds of free constructions are likely to feature more and more in tt-geometry and related fields, and it would be interesting to see if our methods can be generalized in that direction as well. 

We note that in a more classical direction, free rigid categories have already been considered in the context of additive $1$-categories, cf. the already-mentioned \cite[Example 1.26]{delignemilne} and \cite[Définition 10.2]{deligne}. Our work generalizes these results beyond the case of a single free generator, but also places it in the context of $\infty$-categories: these works are $1$-categorical, and in particular do not see the $S^1$-action on dimensions that is essential in our results and in some applications. 

Back in the $\infty$-categorical context, Neuhauser also studies rigid commutative algebras in abstract symmetric monoidal $(\infty,2)$-categories, where he describes the universal symmetric monoidal $(\infty,2)$-category on a rigid commutative algebra. While the words may sound similar (``free'' and ``universal''), our work is quite different from Neuhauser's \cite{neuhauser}: therein he describes the universal rigid commutative algebra, which lives in a specific, fixed, symmetric monoidal \twocategory, while we fix the symmetric monoidal \twocategory and ask for free rigid commutative algebras therein. The contents are related, and the author originally imagined being able to prove \Cref{thm:maintext} by using \cite{neuhauser}. Even in the current proof architecture, the main result of \emph{loc.cit.} was used in an earlier version of \Cref{cons:cocone}, and, as mentioned in \Cref{rmk:En}, might feature in a discussion of the $\EE_n$-monoidal case. 
 
 Finally, a less obvious relation is to the author's \cite{KNS}. In that work, I heavily use a description of free rigid $C$-algebras for $C$ itself a (locally) rigid commutative algebra in $\PrL_\st$.  The present paper originated as an attempt to understand the results of Barkan and Steinebrunner I was using therein, as well as to minimize the necessary translation between the $\V$-enriched and $\V$-tensored worlds. The results I prove here do allow to minimize this translation, since they are directly in the language of rigid commutative algebras in $\Mod_\V(\PrL)$, and not of enriched symmetric monoidal \categories. 
\subsection*{Sectionwise outline}
\Cref{section:prelim} gathers preparations for the main construction and result: some basics about left adjoints in \twocategories, some recollections about rigid commutative algebras, generalized from the setting of \cite{LocRig} to the setting of a symmetric monoidal \twocategory, and finally we record some elementary interactions between oplax colimits and left adjoints. 

In \Cref{section:cons}, we prove that oplax colimits of lax symmetric monoidal functors are, in certain cases, rigid commutative algebras. We defer the precise construction of these oplax colimits as commutative algebras (\textit{à la} Thom spectrum) to the technical appendix \Cref{app:construction}, and focus in that section on the proof that they are actually rigid.  

In \Cref{section:proof}, we prove that the construction from \Cref{section:cons} actually provides a construction of \emph{free} rigid commutative algebras.

  \Cref{section:calculations} extracts some useful calculations in free rigid commutative algebras, such as endomorphisms of the unit or generalized mapping objects.

  Finally, \Cref{section:locrig} deduces some consequences of our formula for the theory of locally rigid commutative algebras and their rigidifications, in the sense of \cite{LocRig}. 
\subsection*{Conventions}
We work with \categories as extensively developed in Lurie's \cite{HTT,HA}, and mainly try to follow the conventions therein, except that we use ``anima'' in place of ``spaces'' or ``$\infty$-groupoids'', and denote their \category by $\An$. 

We take the nowadays standard convention that in the main text, they are called ``categories'', and ordinary categories are viewed as categories with truncated mapping anima. We do \emph{not} use the standard term ``$1$-category'', since we will use it in opposition to ``$2$-category'', which we use to refer to \twocategories. 

For this, we note that by \cite{loubaton2025squares}, all models and all notions within those models have been compared, and we work with $2$-categories in a model-agnostic way. 

Besides that, some specific conventions are: 
\begin{enumerate}
    \item We write $\Cat$ for the category of small categories, and $\widehat{\Cat}$ for the (very large) category of large categories; 
    \item We write $\widehat{\Cat}^\colim$ for the non-full subcategory of $\widehat{\Cat}$ spanned by categories with all small colimits and colimit-preserving functors, and $\PrL\subset\widehat{\Cat}^\colim$ for the full subcategory spanned by presentable categories; 
    \item If $C$ is a symmetric monoidal ($2$-)category, we write $\CAlg(C), \Mod_A(C), \one_C$ for, respectively, the ($2$-)category of commutative algebras in $C$, the category of modules over a commutative algebra $A$ in $C$, and the unit of $C$, with the exception of $\CAlg(C)$, which we will exclusively view as a $1$-category, even if $C$ is a $2$-category;
    \item We write $\Psh$ for presheaf categories; 
    \item We write $\Map$ for mapping anima in categories, and $\Hom$ for mapping categories in $2$-categories;
     \item We follow \cite{heine} for the comparison between enriched and (weakly) tensored cateories. 
\end{enumerate}
\subsection*{Acknowledgements}
I am extremely indebted to Jan Steinebrunner for numerous conversations related to his work with Barkan \cite{Jan}, which ultimately led to the present work. 

I am thankful to Leor Neuhauser for allowing me to include his proof of \Cref{prop:laxstrict}, and to Markus Zetto for helpful conversations about $\PrL$-tensoring. I am grateful to Fernando Abell\'an, Thomas Blom, Bastiaan Cnossen and Thorger Gei{\ss} for help with the details surrounding \Cref{cons:colim} (which feature in \Cref{app:construction}). Finally, I also had interesting conversations with Thomas Nikolaus and Phil Pützstück related to the subject matter.

This research was funded by the Deutsche Forschungsgemeinschaft (DFG, German Research Foundation) – Project-ID 427320536 – SFB 1442, as well as under Germany's Excellence Strategy EXC 2044/2 –390685587, Mathematics Münster: Dynamics–Geometry–Structure.
\section{Preliminaries}\label{section:prelim}
\subsection{Left adjoints}\label{section:ladj}
Recall that in a $2$-category $\B$, there is a notion of left and right adjoints, defined using units, counits and triangle identities, see e.g. \cite{riehlverity} and \cite{haugseng} for a discussion in the context of $(\infty,2)$-categories. 
\begin{ex}
    In $\B=\Cat$, the notions of left and right adjoints specialize to the usual notions.
\end{ex}
\begin{ex}
    Since $\PrL\to\widehat{\Cat}$ is a $2$-functor, it preserves adjoints. Hence, for a functor $f:C\to D$ in $\PrL$ to be a left adjoint \emph{in} $\PrL$, we need its ordinary right adjoint (which always exists, by the adjoint functor theorem) to also lie in $\PrL$, i.e. to preserve colimits. 

    Since the forgetful functor to $\widehat{\Cat}$ is locally fully faithful, it follows that in $\B=\PrL$, left adjoints are functors whose right adjoint also preserves colimits. 
\end{ex}
\begin{ex}
    For $\V\in\CAlg(\PrL)$, left adjoints in $\Mod_\V(\PrL)$ are the internal left adjoints of \cite[Definition 1.9]{Dbl}.
\end{ex}

In this short section, we provide dual convenient criteria to check when adjoints exist. 

The first is the following: 
\begin{lm}\label{lm:intladj0}
 Let $\B$ be a $2$-category and $f:b_0\to b_1$ a morphism in $\B$. Suppose that for every $Z$, $f\circ -:\Hom_\B(Z,b_0)\to \Hom_\B(Z,b_1)$ admits a right adjoint. In this case, $f$ admits a right adjoint if and only if for any $g: Z'\to Z$, the following square is horizontally right adjointable: 
 \[\begin{tikzcd}
	{\Hom_\B(Z,b_0)} & {\Hom_\B(Z,b_1)} \\
	{\Hom_\B(Z',b_0)} & {\Hom_\B(Z',b_1)}
	\arrow[from=1-1, to=1-2]
	\arrow[from=1-1, to=2-1]
	\arrow[from=1-2, to=2-2]
	\arrow[from=2-1, to=2-2]
\end{tikzcd}\]
\end{lm}
\begin{proof}
The assumption tells us that $f_*:\Hom_\B(-,b_1)\to \Hom_\B(-,b_0)$ is a left adjoint in $\Fun_2(\B\op,\Cat)$, cf. \cite[Theorem 4.6]{haugseng}.

By the Yoneda lemma, implies directly that $f$ is a right adjoint in $\B\subset \Fun_2(\B\op,\Cat)$.
\end{proof}

\begin{rmk}
 In a strongly locally presentable $2$-category in the sense of \Cref{defn:strlp}, every $f$ satisfies the starting assumption of that dual version. Thus, in some sense, every $f$ is ``externally'' right adjointable, and being an internal left adjoint is a checkable condition. 
\end{rmk}
\begin{rmk}
  When specialized to $\B=\Mod_\V(\PrL)$ this criterion recovers the usual notion of internal left adjoints, cf. \cite[Definition 1.9]{Dbl}, by plugging in $Z= Z'= \V$ but the map $g$ is given by $v\otimes -: \V\to \V$ for $v\in \V$, or by plugging in $Z= \Fun(I\op,\V), Z'= \V$ and $g$ is given by the diagonal map. 
 \end{rmk}
We will also need the following version, which is simply a dual version and proved exactly the same way - we spell it out for the convenience of the reader:
\begin{lm}\label{lm:intladj}
    Let $\B$ be a $2$-category and $f:b_0\to b_1$ a morphism in $\B$. Suppose that for every $Z$, $-\circ f:\Hom_\B(b_1,Z)\to \Hom_\B(b_0,Z)$ admits a left adjoint. In this case, $f$ admits a right adjoint if and only if for any $g: Z\to Z'$, the following square is horizontally left adjointable: 
    \[\begin{tikzcd}
	{\Hom_\B(b_1,Z)} & {\Hom_\B(b_0,Z)} \\
	{\Hom_\B(b_1,Z')} & {\Hom_\B(b_0,Z')}
	\arrow["{-\circ f}", from=1-1, to=1-2]
	\arrow["{g\circ-}"', from=1-1, to=2-1]
	\arrow["{g\circ -}", from=1-2, to=2-2]
	\arrow["{-\circ f}"', from=2-1, to=2-2]
\end{tikzcd}\]
\end{lm}
 We will use this version of the criterion lemma in conjunction with: 
\begin{lm}\label{lm:jtlycoco}
    Let $f_i: A\to B_i$ be a jointly conservative family of functors. Suppose each $f_i$ admits a left adjoint $f_i^L$. Let $\eta: F\to G$ be a transformation between functors $F,G: A\to C$ admitting right adjoints. If for all $i$, $\eta f_i^L$ is an equivalence, then $\eta$ is an equivalence. 
 \end{lm}
 \begin{proof}
     It follows that the transformation $f_i \circ G^R\to f_i\circ F^R$ induced by $\eta$ is an equivalence. Hence by assumption, $G^R\to F^R$ is an equivalence, and hence, $F\to G$ is an equivalence. 
 \end{proof}

\subsection{Rigid algebras}\label{section:rig}
Let $\B$ be a symmetric monoidal $2$-category. We consider the following definition, originally due to Gaitsgory--Rozenblyum \cite[Chapter 1, §9]{gaitsroz}, studied in this generality in \cite{neuhauser} and in the case of $\B=\Mod_\V(\PrL)$ in \cite{LocRig}:
\begin{defn}
    A commutative algebra $A\in\CAlg(\B)$ is \emph{rigid} if the unit map $\one_\B\to A$ has a right adjoint, and the multiplication map $A\otimes A\to A$ has a right adjoint for which the lax $A$-module structure is strict.

    We let $\CAlg^\rig(\B)\subset\CAlg(\B)$ denote the full subcategory spanned by rigid commutative algebras. 
\end{defn}
One of the motivating examples is the following: 
\begin{ex}
    Let $C$ be a small symmetric monoidal category. If all objects of $C$ are dualizable, then $\Psh(C)$ with the Day convolution structure is rigid in $\B=\PrL$. Conversely, if $\Psh(C)$ is rigid, then every object in $C$ is dualizable in the idempotent-completion of $C$. 
\end{ex}
Thus, up to idempotent-completion, which in this case is very easy to control, if we can describe free rigid commutative algebras, we can in particular describe free symmetric monoidal categories with duals. 

Recall now that the ($1$D) cobordism hypothesis describes the free symmetric monoidal cateory on a dualizable object. It was conjectured by Baez and Dolan \cite{baezdolan}, with also a higher dimensional version, which we shall not discuss. Lurie sketched a proof of this conjecture in all dimensions \cite{luriecob}, and Harpaz gave a full proof in dimension $1$ \cite{harpazcob}. The work of Barkan and Steinebrunner mentioned earlier also provides a new complete proof (and generalizations) \cite{Jan}. 
\begin{thm}[\cite{harpazcob},\cite{luriecob}]
The free symmetric monoidal \category on a dualizable object is the category $\Cob^{\mathrm{1d,or}}$ of 1-dimensional oriented cobordisms between oriented 0-manifolds.
\end{thm} 
For simplicity, we remove the adjectives from the notation since no other cobordism category will appear in this paper:
\begin{nota}\label{nota:cob}
    Let $\Cob$ denote the category of framed $1$-dimensional cobordisms, cf. \cite{harpazcob}. 
\end{nota}

In this article, we aim to \emph{deduce} a certain generalization of this $1$D cobordism from this version, unlike Barkan and Steinebrunner who prove everything from scratch. In particular, our results do not actually pertain to the cobordism category: rather, most of what we prove would remain unchanged if we changed $\Cob$ to ``the free symmetric monoidal category on a dualizable object'', except maybe for the precise calculations at the end of \Cref{section:calculations}. 

Changing $\PrL$ with $\Mod_\V(\PrL)$ for some $\V\in\CAlg(\PrL)$ is our actual motivating example, when plugging in $\V$-presheaf categories. Via the relationship between $\V$-enriched categories and $\V$-modules \cite{heine}, we are actually describing (up to Cauchy completion) free symmetric monoidal $\V$-enriched categories with duals. However, our results are more general, and thus apply, e.g., to dualizable presentable stable categories that are not compactly generated.  

In any case, we will need a few prerequisites about rigid commutative algebras, which we now record. They can all be found in \cite{LocRig} in the case $\B=\Mod_\V(\PrL)$, and the generalization is rather mild. 

As a basic input, we will use the following duality result:
\begin{lm}\label{lm:dualrig}
Let $\B$ be a symmetric monoidal $2$-category and $A\in\CAlg^\rig(\B)$. The underlying object of $A$ is dualizable, with coevaluation given by $$\one_\B\xrightarrow{\eta}A\xrightarrow{m^R}A\otimes A $$
and evaluation 
$$A\otimes A\xrightarrow{m} A\xrightarrow{\eta^R}\one_\B$$
where $m$ is the multiplication of $A$, and $\eta$ the unit. 
\end{lm}
\begin{proof}
    The triangle identites follow directly from the linearity of $m^R$, see \cite[Proposition 2.17]{HSSS} or, in this generality, \cite[Corollary 3.15, Definition 3.9]{neuhauser}. 
\end{proof}

Our first real goal is to establish a convenient criterion to check fully faithfulness of commutative algebra maps between rigid commutative algebras. Along the way, the following result is useful:
\begin{prop}[{See \cite[Proposition 4.18]{LocRig}}]\label{prop:ladjunder}
    Let $\B$ be a symmetric monoidal $2$-category, and $A\in\CAlg^\rig(\B)$. 

    Let $f:M\to N$ be a map of $A$-modules. If $f$ admits a right adjoint in $\B$, then it also has one in $\Mod_A(\B)$.
\end{prop}
We will deduce this from a stronger result, the proof of which we learned from Leor Neuhauser (whom we thank for letting us include the argument here - any mistake is the author's): 
\begin{prop}\label{prop:laxstrict}
    Let $\B$ be a symmetric monoidal $2$-category, and $A\in\CAlg^\rig(\B)$. Any lax $A$-linear map between $A$-modules is strictly $A$-linear. 
\end{prop}

We begin with the following: 
\begin{lm}\label{lm:actionladj}
    Let $\B$ be a symmetric monoidal $2$-category and $A$ a rigid commutative algebra in $\B$. For any $A$-module $M$, the action map $\rho: A\otimes M\to M$ is left adjoint in $\Mod_A(\B)$. 

    Its right adjoint is, on underlying objects, equivalent to its duality mate. 
\end{lm}
\begin{proof}
    Up to embedding $\B$ in $\Fun(\B\op,\Cat)$, we may assume $\B$ is cocomplete and its tensor product commutes with $2$-colimits in each variable. 

    Now, the multiplication map $A\otimes A\to A$ is an $A$-bimodule left adjoint. Since $\otimes_A$ is a $2$-functor, we find that $(A\otimes A)\otimes_A M\to A\otimes_A M$ is an $A$-linear left adjoint - but this map is equivalent to the action map, so this proves the result. 

    To verify what the underlying map of the right adjoint is, we first observe that the duality mate is given by the following composite, by \Cref{lm:dualrig}: $$M\xrightarrow{\eta \otimes M}A\otimes M\xrightarrow{m^R\otimes M} A\otimes A\otimes M \xrightarrow{A\otimes\rho} A\otimes M$$ 

    With this in mind, the claim follows from the proof above and the following commutative diagram: 
    \[\begin{tikzcd}
	M & {A\otimes M} & {A\otimes A\otimes M} \\
	& {A\otimes_A M} & {(A\otimes A)\otimes_A M} & {A\otimes M}
	\arrow["{\eta\otimes M}", from=1-1, to=1-2]
	\arrow["\simeq"', from=1-1, to=2-2]
	\arrow["{m^R\otimes M}", from=1-2, to=1-3]
	\arrow[from=1-2, to=2-2]
	\arrow[from=1-3, to=2-3]
	\arrow["{A\otimes \rho}", from=1-3, to=2-4]
	\arrow["{m^R\otimes_AM}"', from=2-2, to=2-3]
	\arrow["\simeq"', from=2-3, to=2-4]
\end{tikzcd}\]
\end{proof}
\begin{proof}[Proof of \Cref{prop:laxstrict}]
First, we note that since we already have the coherent lax $A$-linear structure on $f$, the claim that it is strictly linear can be checked in the homotopy $2$-category of $\B$, so that there are very few coherences to check for every claim. 

    Consider now the symmetric monoidal $2$-category $\Fun_\oplax(\Delta^1,\B)$, given by the pointwise symmetric monoidal structure. The identity map $\id_A$ forms a rigid commutative algebra therein, and the lax $A$-linear map $f: M\to N$ corresponds to a module in $\Fun_\oplax(\Delta^1,\B)$ over $\id_A$ (the $\lax$ became an $\oplax$ because of non-commutativity of the Gray tensor product).

    Thus, by \Cref{lm:actionladj}, the action map $\id_A\otimes f\to f$ has a right adjoint in $\Fun_\oplax(\Delta^1,\B)$. By \cite[Theorem 4.6]{haugseng}, this right adjoint must be a strict map of arrows, that is, the mate 
    \[\begin{tikzcd}
	{A\otimes M} & {A\otimes N} \\
	M & N
	\arrow["{A\otimes f}", from=1-1, to=1-2]
	\arrow[from={0.2}{0.8}, Rightarrow, from=1-1, to=2-2]
	\arrow["{\rho_M^R}", from=2-1, to=1-1]
	\arrow["f"', from=2-1, to=2-2]
	\arrow["{\rho_N^R}"', from=2-2, to=1-2]
\end{tikzcd}\] 
of the lax naturality square commutes strictly. Since the vertical maps are given by duality mates of $\rho_M, \rho_N$ respectively, see \Cref{lm:actionladj}, it follows that the original square also commuted strictly. 
\end{proof}

\begin{proof}[Proof of \Cref{prop:ladjunder}]
     By \cite[Proposition 2.22]{neuhauser}, the right adjoint $f^R$ is canonically a right adjoint in $A$-modules and lax $A$-linear maps. By \Cref{prop:laxstrict}, it is therefore also an adjoint in strict $A$-linear maps, as was to be shown. 
\end{proof}
The following is (a priori weaker than, but in light of \Cref{prop:ladjunder}, equivalent to) \cite[Proposition 3.20]{neuhauser} (see also the second half of \cite[Observation 4.51]{LocRig}): 
\begin{lm}\label{lm:existadj}
Let $\B$ be a symmetric monoidal $2$-category and $f:A\to B$ a morphism in $\CAlg^\rig(\B)$. $f$ admits a right adjoint in $\B$.
\end{lm}
As a corollary, we obtain the desired fully faithfulness criterion: 
\begin{cor}\label{cor:rigff}
    Let $\B$ be a symmetric monoidal $2$-category, and let $f:A\to B$ be a map in $\CAlg^\rig(\B)$. By \Cref{lm:existadj}, it has a right adjoint. 

    In this situation, $f$ is fully faithful, i.e. the unit $\id_A\to f^Rf$ is an equivalence, if and only if it is fully faithful ``on the unit'', i.e. $\eta_A\to f^Rf\eta_A$ is an equivalence, where $\eta_A: \one_\B\to A$ is the unit. 
\end{cor}
\begin{proof}
    $f^Rf\simeq f^R(\mu_B f\eta_A \otimes f)$. By \Cref{prop:ladjunder}, $f^R$ is $A$-linear and hence the latter is equivalent to $\mu_A\circ (f^R f\eta_A \otimes  \id_A) $. 

    By assumption, this is equivalent to $\mu_A\circ (\eta_A\otimes\id_A)\simeq \id_A$. 

    One can either trace through the equivalences and check that the relevant map $\id_A\to f^Rf$ is an equivalence, or to avoid having to do this, use \cite[Lemma 3.3.1]{AmbiHeight} to deduce that if $f^Rf\simeq \id_A$, then the unit map $\id_A\to f^Rf$ is an equivalence. 
\end{proof}
There is a dual version which is also convenient, and is a generalization of \cite[Proposition 4.35]{LocRig}: 
\begin{cor}\label{cor:rigff2}
    Let $\B$ be a symmetric monoidal $2$-category, and let $f:A\to B$ be a map in $\CAlg^\rig(\B)$. By \Cref{lm:existadj}, it has a right adjoint. 

    In this situation, $f$ is fully faithful if and only if the map $\eta_A^R\to \eta_A^Rf^Rf\simeq \eta_B^Rf$ is an equivalence, i.e. $f$ is ``fully faithful out of the unit''. 
\end{cor}
\begin{proof}
    By \Cref{lm:dualrig}, $A\otimes A\xrightarrow{m}A\xrightarrow{\eta_A^R}\one_\B$ is a duality pairing. Thus, $\id_A\to f^R f$ is an equivalence if and only if $\eta_A^R m\to \eta_A^R m\circ (\id_A\otimes f^R f)$ is an equivalence. 
    
    But $f^Rf$ is $A$-linear so we can rewrite the target as $\eta_A^R\circ f^R f \circ m$ and one then verifies that the map is $(\eta_A^R\to \eta_A^R\circ f^R f)\circ m$, so that the claim follows. 
\end{proof}
We conclude this section by recalling a notation: 
\begin{nota}\label{nota:dbl}
    Let $\B$ be a symmetric monoidal $2$-category. We let $\B^\dbl$ denote the sub-$1$-category of $\B$ spanned by dualizable objects and left adjoints between them. 
\end{nota}
\begin{ex}
    When $\B=\Mod_\V(\PrL)$ for some $\V\in\CAlg(\PrL)$, $\B^\dbl$ is the $\Mod_\V(\PrL)^\dbl$ from \cite{Dbl,LocRig}.
\end{ex}

\subsection{Oplax colimits}
In this section, we prove some results describing the interaction of oplax colimits and left adjoints in sufficiently cocomplete $2$-categories. 

We start off with some definitions:
\begin{defn}\label{defn:strlp}
    A $2$-category $\B$ is called locally cocomplete if for all $b_0,b_1\in\B, \Hom_\B(b_0,b_1)$ is cocomplete. 

    It is called strongly locally cocomplete if furthermore, each composition functor $\Hom_\B(b_0,b_1)\times\Hom_\B(b_1,b_2)\to \Hom_\B(b_0,b_2)$ is cocontinuous in each variable. 

    It is called locally presentable\footnote{This is an unfortunate clash in terminology with the $1$-categorical notion of ``locally presentable''. Since in the $\infty$-category literature, the $1$-categorical notion has been rebranded ``presentable'', we hope no confusion will arise.} if each $\Hom_\B(b_0,b_1)$ is presentable, and strongly locally presentable if it is locally presentable and strongly locally cocomplete. 


\end{defn}
\begin{nota}
    Let $\Prof\subset\PrL$ denote the full subcategory spanned by presheaf categories. 
\end{nota}
\begin{obs}
    Suppose $\B$ is a strongly locally cocomplete $2$-category. Then $\B$ is tensored over $\Prof\subset\widehat{\Cat}^\colim$. 
\end{obs}
\begin{proof}
    Indeed, for $I$ a small category, and $b\in\B$, we claim that $I\otimes b$ is a tensor over $\Psh(I)$. For this, simply note that $\Map_{\widehat{\Cat}^\colim}(\Psh(I),\Hom_\B(b,b_1))\simeq \Map_{\widehat{\Cat}}(I,\Hom_\B(b,b_1))\simeq \Map_\B(I\otimes b,b_1)$. 
\end{proof}
\begin{defn}
    A (possibly locally large) $2$-category is called $2$-cocomplete if its underlying $1$-category is cocomplete and it is\footnote{This is a property, since it is already enriched.} tensored over $\Cat\subset\widehat{\Cat}$. 
\end{defn}
\begin{cor}\label{cor:Prtensor}
    Suppose $\B$ is a strongly locally cocomplete, $2$-cocomplete $2$-category. In this case, $\B$ is tensored over $\PrL$.
\end{cor}
For the proof, we will need: 
\begin{lm}
 The inclusion $\PrL\subset\widehat{\Cat}^\colim$ preserves small colimits.  
\end{lm}
\begin{proof}
    Let $C_\bullet: I\to \PrL$ be a small diagram, and $\kappa$ a regular cardinal such that $C_\bullet$ lands in $\PrL_\kappa$. 

    We know that $\colim_I C_i\simeq \Ind_\kappa(\colim_I C_i^\kappa)$ where the colimit inside is taken in $\Cat^{\rex(\kappa)}$.

    Now colimits in $\Cat^{\rex(\kappa)}$ are also colimits in $\widehat{\Cat}^{\rex(\kappa)}$. Thus, for any cocomplete $D$ we find $$\Fun^L(\colim_i C_i,D)\simeq \Fun^{\rex(\kappa)}(\colim_i C_i^\kappa,D)\simeq \lim_i \Fun^{\rex(\kappa)}(C_i^\kappa,D)\simeq \lim_i\Fun^L(C_i,D)$$
    as was to be shown. 
\end{proof}
\begin{proof}[Proof of \Cref{cor:Prtensor}]
    
    Since $\B$ is cocomplete and tensored over $\Cat$, The collection of objects in $\widehat{\Cat}^\colim$ for which $\B$ admits tensors is closed under colimits. By the previous lemma and \cite[Lemma 7.14]{ragimovschlank}, this implies that $\B$ admits tensors for $\PrL$.  
\end{proof}
We can now prove the first main result of this section:
\begin{prop}\label{prop:incladj}
    Let $\B$ be a strongly locally cocomplete, $2$-cocomplete $2$-category. In this case, for any small category $I$ and diagram $F: I\to \B$, each of the oplax cocone morphisms $F(i)\to \colim^\oplax_I F$ is a left adjoint in $\B$. 
\end{prop}
\begin{rmk}
  Here, and in \Cref{prop:colimitladj}, we do \emph{not} assume, and it is crucial for our later applications, that the morphisms $F(i)\to F(j)$ are themselves left adjoints in $\B$. 
\end{rmk}
\begin{proof}
    By \Cref{cor:Prtensor}, the diagram $F$ factors as $I\xrightarrow{Y}\Fun(I\op,\PrL)\to \B$ where $Y$ is the $\PrL$-linearized Yoneda embedding. Thus, we may assume $\B\simeq \Fun(I\op,\PrL)$, and $F=Y$.  

    In this case, the oplax cocone morphisms are of the form $\An[\Map_I(-,i)]\to \An[(I_{-/})\op]$ given by left Kan extension along $\Map_I(-,i)\to (I_{-/})\op$. Thus the right adjoint preserves colimits for each $-\in I$ and lives in $\PrL$. We are left with checking that for each $f:j\to j'$ in $I$, the following diagram is horizontally right adjointable:
    \[\begin{tikzcd}
	{\An[\Map_I(j',i)]} & {\An[(I_{j'/})\op]} \\
	{\An[\Map_I(j,i)]} & {\An[(I_{j/})\op]}
	\arrow[from=1-1, to=1-2]
	\arrow["{(-\circ f)_!}", from=1-1, to=2-1]
	\arrow["{(-\circ f)_!}", from=1-2, to=2-2]
	\arrow[from=2-1, to=2-2]
\end{tikzcd}\]

Unwinding the pointwise Kan extension formula, this amounts to checking that for every map $g:j\to i$, the map from the fiber $\Map_I(j',i)\times_{\Map_I(j,i)}\{g\}$ to the lax fiber $(I_{j'/})\times_{(I_{j/})} (I_{j/})_{/g}$ is colimit-cofinal. This is a standard argument: the lax fiber in question is the category of $j'\to x$'s equipped with a factorization of $g$ through $j\to j'\to x$, and the inclusion of the fiber has a left adjoint given by composing the arrows $j'\to x\to i$, thereby proving cofinality. 
\end{proof}
From the above, we actually deduce how to compute with these oplax colimits: 
\begin{cor}\label{cor:mapformula}
    Let $\B$ be a strongly locally cocomplete, $2$-cocomplete $2$-category, and let $F:I\to \B$ be a diagram. For $j\in I$, let $\iota_j: F(j)\to \colim^\oplax_I F$ be the canonical map. By \Cref{prop:incladj}, it has a right adjoint and for every $i\in I$, there is a canonical cocone $F(f)\to \iota_j^R\iota_i$ indexed my $\Map_I(i,j)$. 

    This is a colimit cocone, i.e. $\colim_{\Map_I(i,j)}F(f)\simeq \iota_j^R\iota_i$, and in the universal example $\B=\Fun(I\op,\PrL)$, both are terminal in the appropriate $\Hom$-category.
\end{cor}
\begin{proof}
    For any oplax cocone $p:F\to C$, we have a lax commutative diagram: 
    \[\begin{tikzcd}
	{\Map_I(i,j)} & {\Hom_\B(F(i),F(j))} \\
	{*} & {\Hom_B(F(i),C)}
	\arrow["F", from=1-1, to=1-2]
	\arrow[from=1-1, to=2-1]
	\arrow[from={0.2}{0.8}, Rightarrow, from=1-2, to=2-1]
	\arrow["{p_j\circ -}", from=1-2, to=2-2]
	\arrow["{p_i}"', from=2-1, to=2-2]
\end{tikzcd}\]
which provides a natural transformation $p_j\circ F(f)\to p_i$. For our oplax colimit cocone, since $p_j=\iota_j$ has a right adjoint, we obtain the desired cocone $F(f)\to \iota_j^R\iota_i$ in $\Hom_\B(F(i),F(j))$. 

To prove this is a colimit cocone, i.e. that $\colim_{\Map_I(i,j)}F(f)\simeq \iota_j^R\iota_i$, we reduce to the universal case $\Fun(I\op,\PrL)$ as in the proof of \Cref{prop:incladj}.

In that case, we compute that both sides are actually terminal objects in $$\Hom_{\Fun(I\op,\PrL)}(\An[\Map_I(-,i)],\An[\Map_I(-,j)])\simeq \An[\Map_I(i,j)],$$ and so any map between them is an equivalence. 

    For the left hand side, this is the fact that the colimit of the Yoneda embedding $C\to \An[C]$ is terminal for any $C$, and for the right hand side we simply consider the following diagram $\An[\Map_I(i,i)]\xrightarrow{\iota_i} \An[(I_{i/})\op]\xrightarrow{\iota_j^R} \An[\Map_I(i,j)]$. By design, it sends $\id_i$ to the identity of $i$, viewed as an object in $(I_{i/})\op$, which is clearly terminal. Since $\iota_j^R$ is a right adjoint, it preserves terminal objects, and so the claim is proved. 
\end{proof}
\begin{prop}\label{prop:colimitladj}
    Let $\B$ be an strongly locally presentable, $2$-cocomplete $2$-category. For any small diagram $F:I\to \B$ and oplax cocone $F\to b$, the induced map $\colim^\oplax_I F\to b$ is a left adjoint in $\B$ if and only if each $F(i)\to b$ is a left adjoint. 
\end{prop}
We begin with:
 \begin{lm}\label{lm:limladj}
     Let $G:I\to \PrL$ be a diagram and let $C\to G$ be a lax cone, where $C$ is also presentable. Assume each $C\to G(i)$ has a left adjoint. In this case, $C\to \lim^\lax_I G$ also has a left adjoint. 
\end{lm}
\begin{proof}
By presentability, it suffices to prove that $C\to \lim^\lax_I G$ preserves limits and is accessible. Limits in lax limits of complete categories are underlying, thus the assumption that each $C\to G(i)$ admits a left adjoint guarantees that $C\to \lim^\lax_I G$ preserves limits. 

Furthermore, each $C\to G(i)$ is $\kappa$-accessible for some fixed regular $\kappa$ (pick one $\kappa_i$ for each $i$ and let $\kappa$ be a regular upper bound of the $\kappa_i$'s). Since the transition functors in $G$ preserve $\kappa$-filtered colimits (in fact, all colimits), these are also computed underlying in the lax limit, and thus $C\to \lim^\lax_I G$ also preserves them. 

By the adjoint functor theorem, $C\to \lim^\lax_I G$ admits a left adjoint. 
 \end{proof}
\begin{proof}[Proof of \Cref{prop:colimitladj}]
    By \Cref{prop:incladj}, each $F(i)\to \colim^\oplax_I F$ is a left adjoint, so this shows ``Only if''. 

For ``If'', we use \Cref{lm:intladj}.
First, fix $Z\in \B$ and consider the map $\Hom_\B(b,Z)\to \Hom_\B(\colim^\oplax_I F,Z)\simeq \lim^\lax_I \Hom_\B(F,Z)$. By \Cref{lm:limladj}, it has a left adjoint since $\B$ is strongly locally presentable. 

By \Cref{lm:intladj}, it therefore suffices to check that for each $g:Z\to Z'\in \B$, the following square is horizontally left adjointable: 
 \[\begin{tikzcd}
 	{\Hom_\B(b,Z)} & {\Hom_\B(\colim^\oplax_I F,Z)} \\
	{\Hom_\B(b,Z')} & {\Hom_\B(\colim^\oplax_I F,Z')}
	\arrow[from=1-1, to=1-2]
	\arrow[from=1-1, to=2-1]
	\arrow[from=1-2, to=2-2]
	\arrow[from=2-1, to=2-2]
\end{tikzcd}\]
 By definition of oplax colimits, the family of maps $\iota_i^*: \Hom_\B(\colim^\oplax_I F,Z)\to \Hom_\B(F(i),Z)$ (resp. $Z'$) is jointly conservative, where $\iota_i: F(i)\to \colim^\oplax_I F$ is the inclusion. Each $\iota_i^*$ has a left adjoint given by $(\iota_i^R)^*$, and since adjointability is a question about a natural transformation between functors $\Hom_\B(\colim^\oplax_I F,Z)\to \Hom_\B(b,Z')$, we are in a position to apply \Cref{lm:jtlycoco}. Indeed, the two relevant functors have right adjoints: one of them is definitionally a left adjoint followed by a postcomposition morphism which has an adjoint by strong local presentability, and the other one is a similar composition but in the other order. 

 Thus we may apply \Cref{lm:jtlycoco} and it thus suffices to check adjointability after precomposition by each of the folloiwing maps: 
\[\begin{tikzcd}
	{\Hom_\B(b,Z)} & {\Hom_\B(\colim^\oplax_IF,Z)} & {\Hom_\B(F(i),Z)} \\
	{\Hom_\B(b,Z')} & {\Hom_\B(\colim^\oplax_IF,Z')}
	\arrow[from=1-1, to=1-2]
	\arrow[from=1-1, to=2-1]
	\arrow[shift left=3, dashed, from=1-2, to=1-1]
	\arrow[from=1-2, to=2-2]
	\arrow["(\iota_i^R)^*",from=1-3, to=1-2]
	\arrow[from=2-1, to=2-2]
	\arrow[shift left=3, dashed, from=2-2, to=2-1]
\end{tikzcd}\]

But now, each of of the $\iota_i^R$ is an internal right adjoint, hence the following squares are all horizontally left adjointable: 
\[\begin{tikzcd}
	{\Hom_\B(\colim^\oplax_IF,Z)} & {\Hom_\B(F(i),Z)} \\
	{\Hom_\B(\colim^\oplax_IF,Z')} & {\Hom_\B(F(i),Z')}
	\arrow["{\iota_i^*}", from=1-1, to=1-2]
	\arrow[from=1-1, to=2-1]
	\arrow[from=1-2, to=2-2]
	\arrow["{\iota_i^*}"', from=2-1, to=2-2]
\end{tikzcd}\] 
Therefore, adjointability of the previous square after composition with $(\iota_i^R)^*$ is equivalent to adjointability of the outer square in the following diagram: 
\[\begin{tikzcd}
	{\Hom_\B(b,Z)} & {\Hom_\B(\colim^\oplax_IF,Z)} & {\Hom_\B(F(i),Z)} \\
	{\Hom_\B(b,Z')} & {\Hom_\B(\colim^\oplax_IF,Z')} & {\Hom_\B(F(i),Z')}
	\arrow[from=1-1, to=1-2]
	\arrow[from=1-1, to=2-1]
	\arrow["{\iota_i^*}", from=1-2, to=1-3]
	\arrow[from=1-2, to=2-2]
	\arrow[from=1-3, to=2-3]
	\arrow[from=2-1, to=2-2]
	\arrow["{\iota_i^*}"', from=2-2, to=2-3]
\end{tikzcd}\]
In turn, this one follows from the assumption that each $F(i)\to b$ admits a right adjoint. 
\end{proof}

\section{The construction}\label{section:cons}
\begin{cons}\label{cons:colim}
Let $I$ be a symmetric monoidal category. 
Let $\B$ be a symmetric monoidal $2$-category with oplax colimits indexed by finite powers of $I$, which are compatible with the tensor product. 

We claim that the diagonal functor $\CAlg(\B)\to \Fun^{\lax-\otimes}_\oplax(I,\B)$ has a left adjoint, whose composition with the forgetful functor $\CAlg(\B)\to \B$ is given by the oplax colimit $\colim^\oplax_I$. 

We abusively still denote this left adjoint by $\colim^\oplax_I$. 
\end{cons}
\begin{recoll}\label{recoll:Funoplax}
    We recall here that $\Fun^{\otimes-\lax}_\oplax(I,\B)$ is defined similarly to $\Fun^{\otimes}_\oplax$ in \cite[Section 2]{HSS}, simply replacing $\Map^{\otimes}$ with $\Map^{\otimes-\lax}$. Concretely, it is the $2$-category representing the functor $K\mapsto \Map^{\otimes-\lax}(I,\Fun_\lax(K,\B))$. In particular, its objects are lax symmetric monoidal functors $I\to \B$, and its morphisms are symmetric monoidal oplax transformations. They are oplax transformations, and the monoidality squares 
    \[\begin{tikzcd}
	{F(x)\otimes F(y)} & {G(x)\otimes G(y)} \\
	{F(x\otimes y)} & {G(x\otimes y)}
	\arrow[from=1-1, to=1-2]
	\arrow[from=1-1, to=2-1]
	\arrow[from=1-2, to=2-2]
	\arrow[from=2-1, to=2-2]
\end{tikzcd}\]
also only oplax commute, that is, commute up to a non-necessarily invertible $2$-cell. 
\end{recoll}
The intuition behind \Cref{cons:colim} is that the relevant oplax colimits have ``obvious'' commutative algebra structures, arising as follows: $$\colim^\oplax_I F\otimes \colim^\oplax_I F\simeq \colim^{\oplax}_{I\times I}(F\boxtimes F) \to \colim^{\oplax}_{I\times I} F\circ m_I \to \colim^{\oplax}_I F$$
where the second arrow uses the lax symmetric monoidal structure of $F$, and the last arrow is a Beck--Chevalley arrow arising from the multiplication map $m_I: I\times I\to I$. Making this precise is somewhat tedious given the current $(\infty,2)$-categorical technology, and therefore, in order not to distract from our story, we defer the precise construction to Appendix~\ref{app:construction}.

Instead, this section is devoted to the proof that the above construction produces rigid commutative algebras, and to explain towards the end our candidate for free rigid commutative algebras. To state it precisely, we first introduce a definition: 
\begin{defn}\label{defn:weak2prl}
    A symmetric monoidal $2$-category $\B$ is called weakly $2$-presentably symmetric monoidal if: 
    \begin{enumerate}
        \item it is strongly locally presentable and $2$-cocomplete; 
        \item Its tensor product commutes with $2$-colimits in each variable; 
        \item Its tensor product induces functors $\Hom_\B(b_0,b_1)\times \Hom_\B(b_2,b_3)\to \Hom_\B(b_0\otimes b_2, b_1\otimes b_3)$ which commute with colimits in each variable. 
    \end{enumerate} 
\end{defn}
\begin{ex}
    $\B=\Mod_{\mathcal{V}}(\PrL)$ is weakly $2$-presentably symmetric monoidal for any $\mathcal{V}\in\CAlg(\PrL)$. 
\end{ex}
In the situation of the example above, the main result of this section is a simultaneous generalization of \cite[Example 4.39]{LocRig} and \cite[Proposition 4.2.4]{keidarragimov} (a twisted generalization of the former, and oplax generalization of the latter): 
\begin{thm}\label{thm:consisrig}
    Let $\B$ be a weakly $2$-presentably symmetric monoidal $2$-category, and $I$ a symmetric monoidal category. Suppose each object of $I$ is dualizable. In this case, for any $F:I\to \B$ symmetric monoidal, $\colim^\oplax_I F \in\CAlg(\B)$ is a rigid commutative algebra. 
\end{thm}
To prove this, we use the following lemma, which is an analogue of the classical fact that symmetric monoidal natural transformations are equivalences on dualizable objects:
\begin{lm}\label{lm:dualadj}
    Let $I$ be a symmetric monoidal category, and let $F,G: I\to \B$ be two symmetric monoidal functors. 

    Let $\eta: F\to G$ be a symmetric monoidal oplax natural transformation, and suppose $i \in I$ is dualizable. In this case, $\eta_i$ has a right adjoint given by $$G(i)\xrightarrow{F(\eta_i)} F(i^\vee)\otimes F(i) \xrightarrow{\eta_{i^\vee}} G(i)\otimes G(i^\vee)\otimes F(i)\xrightarrow{G(\epsilon_i)} F(i)$$ 
\end{lm}
\begin{proof}
    By definition, $\eta$ corresponds to a symmetric monoidal functor $I\to \Lax_{(1)}(\B)$ in the notation of \cite[Section 2.2]{HSS}, and so $\eta_i$ is dualizable in $\Lax_{(1)}(\B)$. By \cite[Lemma 2.4]{HSS}, $\eta_i$ is a left adjoint, and the right adjoint is spelled out in \textit{loc. cit.}. 
    \end{proof}
    
\begin{proof}[Proof of \Cref{thm:consisrig}]
We first observe that by \Cref{prop:incladj}, the unit map $\one_\B \simeq F(\one_I) \to \colim^\oplax_I F$ is a left adjoint in $\B$. 

We now let $C:=\colim^\oplax_I F$ for notational convenience. 

Consider the multiplication map $C\otimes C\to C$. We need to show it is a left $C$-linear left adjoint. Note that as a $C$-linear map, it is given as the oplax colimit of the maps $m_i: C\otimes F(i)\to C$ and thus by \Cref{prop:colimitladj} it suffices to show that each of these is a left adjoint. 

However, we have a symmetric monoidal oplax transformation of lax symmetric monoidal functors $I\to \B$ from $F$ to the constant functor at $C$, by definition. By basechange, this induces a symmetric monoidal oplax transformation of strong symmetric monoidal functors $I\to \Mod_C(\B)$ from $C\otimes F$ to $C$.

By \Cref{lm:dualadj}, it follows that each $C\otimes F(i)\to C$ is a $C$-linear left adjoint, as was to be shown. 

\end{proof}

\begin{ex}\label{ex:univonI}
Let $\B =\Fun(I\op,\PrL)$ for some symmetric monoidal category $I$, and consider the $\PrL$-linearized Yoneda embedding $Y: I\to \B$ as a symmetric monoidal functor. 

In $\colim^\oplax_I Y$, the unit object is terminal. 
\end{ex}
\begin{proof}
The oplax colimit is the presheaf category on the oplax colimit computed in $\Fun(I\op,\Cat)$, namely $i\mapsto (I_{i/})\op$. 

The unit object is the unit object in $\An[(I_{\one/})\op]$ which is the Yoneda image of $\one\in (I_{\one/})\op$ and is therefore terminal. 
\end{proof}

To finally describe our construction, we simply recall the following elementary statement, where $\Cob$ is as in \Cref{nota:cob} and denotes the free symmetric monoidal category on a dualizable object:
\begin{lm}[{\cite[Corollary 2.21]{CCRY}}]\label{lm:oplaxdbl}
    Let $\B$ be a symmetric monoidal $2$-category. Evaluation at the universal dualizable object $+\in\Cob$ induces an equivalence $$\Fun^\otimes_\oplax(\Cob,\B)\xrightarrow{\simeq} \B^\dbl$$ where $\B^\dbl$ is as in \Cref{nota:dbl}. 
\end{lm}
\begin{nota}\label{nota:cobrep}
Let $\B$ be a symmetric monoidal $2$-category. We let $\Cob_\B: \B^\dbl\to\Fun^\otimes_\oplax(\Cob,\B)$ denote the inverse to the equivalence described above.
    
\end{nota}

Fixing $I$ to be the free symmetric monoidal category on a dualizable object, $\Cob$ from \Cref{nota:cob} we thus obtain the following construction:
\begin{cons}\label{cons:Frig}
    We let $F_\B^\rig: \B^\dbl\to\CAlg^\rig(\B)$ denote the corestriction of the composite functor $$\B^\dbl\simeq \Fun^\otimes_\oplax(\Cob,\B)\xrightarrow{\colim^\oplax} \CAlg(\B)$$ using \Cref{lm:oplaxdbl} for the first equivalence, \Cref{cons:colim} for the second map and finally \Cref{thm:consisrig}  to corestrict it to $\CAlg^\rig$.  

    Letting $U:\CAlg^\rig(\B)\to \B^\dbl$ denote the forgetful functor, $F^\rig_\B$ comes with a natural transformation $\eta: \id\to UF^\rig_\B$ given by the canonical map $F(+)\to \colim^\oplax_\Cob F$. 
\end{cons}
$F^\rig_\B$ is our candidate for the ``free rigid commutative algebra'' functor. We will prove in \Cref{section:proof} that this candidate is correct.
\begin{rmk}
    A way to circumvent \Cref{cons:colim} is to use the computation from \Cref{ex:univonI}: in this case, one can give an a priori construction of the commutative algebra structure on the oplax colimit given its explicit nature, and hence deduce it in general by universality. 
\end{rmk}
\begin{rmk}
    We close this section with a remark regarding our application to \cite{KNS}. It turns out that if all we cared about was the results from \textit{loc.cit.}, and one is careful enough in the proofs therein, the existence of $F^\rig_\B$ and the computations from \Cref{section:calculations} are actually sufficient to prove the main results. In other words, ultimately, knowing that $F^\rig_\B(b)$ is actually free on $b$ is not needed for the results of \cite{KNS}.
\end{rmk}
\section{The proof}\label{section:proof}
We now prove our main result, which is:
\begin{thm}\label{thm:maintext}
Let $\B$ be a weakly $2$-presentably symmetric monoidal $2$-category and $x\in \B^\dbl$ a dualizable object. The free rigid commutative algebra on $x$ exists, and is given by the commutative algebra structure from \Cref{cons:colim} on $\colim^\oplax_\Cob \Cob_\B(x)$, $\Cob_\B(x)$ is as in \Cref{nota:cobrep}.  
\end{thm}
\begin{rmk}
    It is likely possible to give an abstract proof of existence of the free rigid commutative algebra by proving ahead of time that $\CAlg^\rig(\B)\to \B^\dbl$ preserves limits, as in \cite[Corollary 4.89]{LocRig}, and proving that both categories are presentable. We have done so in \cite{LocRig} in the case of $\B=\PrL_\st$, but not in the generality here. In any case, we prove existence along the way, and thereby \emph{recover} \cite[Corollary 4.89]{LocRig} without the careful analysis that goes into the proof in \emph{loc. cit.}, which is quite surprising (at least, to the author). 
\end{rmk}
We begin with an elementary construction: 
\begin{cons}\label{cons:cocone}
    Let $A\in \CAlg^\rig(\B)$ be a rigid commutative algebra. The multiplication map $A\otimes A\to A$ is an $A$-linear left adjoint, and hence by \Cref{lm:oplaxdbl} it induces a symmetric monoidal oplax transformation of functors $\Cob\to \Mod_A(\B)$ from $\Cob_{\Mod_A(\B)}(A\otimes A)$ to the constant functor at $A$, $\cst(A)$ (which is $\Cob_{\Mod_A(\B)}(A)$ since $A$ is the unit in $\Mod_A(\B)$). 
    
    By adjunction, it induces a lax symmetric monoidal, oplax transformation of functors $\Cob\to \B$ from $\Cob_\B(A)$ to $\cst(A)$, since $A\otimes\Cob_\B(A)\simeq \Cob_{\Mod_A(\B)}(A\otimes A)$.

    In other words, we obtain a lax symmetric monoidal oplax cocone $\Cob_\B(UA)\to \cst(A)$, and this construction is clearly natural in the pair $(\B,A \in\CAlg^\rig(\B))$. We thus obtain a natural map of rigid commutative algebras $$\epsilon_A: F^\rig_\B(UA)\to A,$$ also natural in $(\B,A)$. 
\end{cons}

\begin{cor}\label{cor:exist}
    Let $\B$ be a weakly $2$-presentably symmetric monoidal $2$-category. The forgetful functor $U: \CAlg^\rig(\B)\to \B^\dbl$ admits a left adjoint, which is a retract of $F^\rig_\B$. 
\end{cor}
\begin{proof}
    By \Cref{cons:Frig} and \Cref{cons:cocone} we have a ``unit'' $\eta: \id\to UF^\rig_\B$ and a ``counit'' $\epsilon : F^\rig_\B U\to \id$. For the duration of the proof, we simply write $F^\rig$. 

    For $x\in \B^\dbl$, we have the following commutative diagram: 
    \[\begin{tikzcd}
	{\Map_{\B^\dbl}(x,U(-))} & {\Map_{\CAlg^\rig(\B)}(F^\rig(x),F^\rig U(-))} & {\Map_{\B^\dbl}(UF^\rig (x), UF^\rig U(-))} \\
	&& {\Map_{\B^\dbl}(x,UF^\rig U(-))} \\
	&& {\Map_{\B^\dbl}(x,U(-))}
	\arrow[from=1-1, to=1-2]
	\arrow["{(\eta U \circ U\epsilon)\circ -}"', from=1-1, to=3-3]
	\arrow[from=1-2, to=1-3]
	\arrow[from=1-3, to=2-3]
	\arrow[from=2-3, to=3-3]
\end{tikzcd}\]

Thus, if we verify one of the triangle identities, namely that the composite $U\xrightarrow{\eta_U} UF^\rig U\xrightarrow{U\epsilon} U$ is equivalent to the identity, we will obtain that the diagonal arrow in this diagram is the identity.

By naturality, we also have a commutative diagram as follows: 
\[\begin{tikzcd}
	{\Map_{\CAlg^\rig(\B)}(F^\rig(x),F^\rig U(-))} & {\Map_{\B^\dbl}(UF^\rig (x), UF^\rig U(-))} \\
	{\Map_{\CAlg^\rig(\B)}(F^\rig(x),-)} & {\Map_{\B^\dbl}(x,UF^\rig U(-))} \\
	{\Map_{\B^\dbl}(UF^\rig(x), U(-))} & {\Map_{\B^\dbl}(x,U(-))}
	\arrow[from=1-1, to=1-2]
	\arrow[from=1-1, to=2-1]
	\arrow[from=1-2, to=2-2]
	\arrow[from=2-1, to=3-1]
	\arrow[from=2-2, to=3-2]
	\arrow[from=3-1, to=3-2]
\end{tikzcd}\]

Therefore, if the triangle identity from above is verified, we will find that $\Map_{\B^\dbl}(x,U(-))$ is a retract of $\Map_{\CAlg^\rig(\B)}(F^\rig(x),-)$. 

Since $\CAlg^\rig(\B)$ is idempotent complete (it is clearly closed under retracts in $\CAlg(\B)$), retracts of corepresentable copresheaves are themselves corepresentable (by a retract of the corepresenting object), and thus the claim will be proved.

We now verify the triangle identity in question: by definition, the map $U(A)\to UF^\rig U(A)$ is the canonical map $U(A)\to \colim^\oplax_\Cob \Cob_\B(UA)$. In turn, the oplax cocone $\Cob_\B(U(A))\to U(A)$ is clearly the identity on $+$, and so the triangle identity follows.
\end{proof}

\begin{nota}
Let $\B$ be weakly $2$-presentably symmetric monoidal. Until the end of the section, we let $L^\rig_\B$ denote the left adjoint to $U:\CAlg^\rig(\B)\to \B^\dbl$, which exists by \Cref{cor:exist}, and we let $i: L^\rig_\B\to F^\rig_\B, p: F^\rig_\B\to L^\rig_\B$ denote the retraction that exists, also by virtue of \Cref{cor:exist}. 
\end{nota}
\begin{proof}[Proof of \Cref{thm:maintext}]
Let $x\in \B^\dbl$. There is a functor of $2$-presentably symmetric monoidal $2$-categories $\overline{x}:\Fun(\Cob\op,\PrL)\to \B$ sending $Y(+)$ to $x$, where $Y$ is the $\PrL$-linearized Yoneda embedding.  

We want to prove that $i,p$ are inverse equivalences. $\overline{x}$ commutes with $F^\rig$ since the latter is given by an oplax colimit, and hence also with $L^\rig$, since $L^\rig$ is a retract of $F^\rig$. It therefore suffices to prove the result in the special case $\B=\Fun(\Cob\op,\PrL)$. For the rest of the proof, $L^\rig$ and $F^\rig$ are considered in this specific case. 

  Now, it suffices to prove that $p$ is fully faithful, since already $p\circ i \simeq \id$. 

    By \Cref{cor:rigff}, it suffices to prove that $p^R$ sends the unit to the unit. 

    In turn, by \Cref{ex:univonI}, the unit in $F^\rig(Y(+))$ is terminal. Since $L^\rig(Y(+))$ is a retract of $F^\rig(Y(+))$, the unit is terminal therein too. But $p^R$ is a right adjoint, so it sends terminal objects to terminal objects, and hences sends the unit to the unit, as was to be shown. 
\end{proof}
\begin{rmk}\label{rmk:adjt}
    A slightly surprising feature of the situation is the following: it is not \textit{a priori} clear to the author that general $2$-colimit-preserving, locally colimit-preserving functors of weakly $2$-presentably symmetric monoidal $2$-categories commute with $L^\rig$, even when they have right adjoints: these right adjoints have no reason to preserve dualizability or rigidity, and so the ``obvious'' adjointability argument fails. 

    However, as a consequence of the theorem, or really, of \Cref{cor:exist}, they do commute with $L^\rig$! This turns out to have interesting consequences, which we explore in \Cref{section:locrig}.  
\end{rmk}
Since we now know that $F^\rig_\B$ is free as a rigid commutative algebra, we can give it a slightly better notation:
\begin{nota}
    We let $\Frrig_\B := F^\rig_\B$.

    When $\B=\Mod_\V(\PrL)$ for some $\V\in\CAlg(\PrL)$, we abusively write $\Frrig_\V$.
\end{nota} 
\section{Some 
calculations}\label{section:calculations}
The goal of this section is to make some calculations with oplax colimits. From our perspective, these calculations are what makes the main theorem actually usable in practice. 

Specifically, we unwind a number of corollaries from \Cref{cor:mapformula}. First, a claculation of endomorphisms of the unit in free rigid commutative algebras, as announced in the introduction. We first define this: 
\begin{defn}\label{defn:End1}
    Let $\B$ be a symmetric monoidal $2$-category and $A$ a commutative algebra in $\B$ such that the unit $\eta:\one_\B\to A$ has a right adjoint. We define the \emph{endomorphism ring object of the unit of} $A$, $\End(\one_A)$, as $\eta^R\eta\in \End_\B(\one_\B)$, which has a canonical commutative algebra structure. 
\end{defn}
We can now state and prove:
\begin{cor}\label{cor:End1}
    Let $\B$ be a weakly $2$-presentably symmetric monoidal $2$-category, and $b\in \B^\dbl$. 

    The endomorphism object of the unit in the free rigid commutative algebra on $b$ is the free commutative algebra in $\End(\one_\B)$ on $\Dim(b)_{hS^1}$. 
\end{cor}
\begin{proof}
    By \Cref{cor:mapformula} the composite $$\one_\B\simeq \Cob_\B(b)(\emptyset) \xrightarrow{\iota_\emptyset} \colim^\oplax_\Cob\Cob_\B(b)\xrightarrow{(\iota_\emptyset)^R} \Cob_\B(b)(\emptyset)\simeq \one_\B$$ is given by $\colim_{f\in\Map_\Cob(\emptyset,\emptyset)} \Cob_\B(b)(f)$. 
    
    We know that $\Map(\emptyset,\emptyset)$ is the free commutative monoid on $BS^1$ in $\An$, i.e. $\coprod_n ((BS^1)^{\times n})_{h\Sigma_n}$ and so we find that the resulting colimit is given, on underlying objects, by $\coprod_n (\Dim(b)_{hS^1})^{\otimes n}_{h\Sigma_n}$.
    
    In the universal example $\B=\Fun(\Cob\op,\PrL)$, this object is terminal in $\End(\one_\B)$ by \Cref{ex:univonI}, and so has a unique commutative algebra structure. Therefore it is the free one in the universal example, and hence this is also the case in general.
\end{proof}

To state our next result, which computes more generally ``mapping objects'', we introduce the following:
\newcommand{\can}{\on{can}}
\begin{nota}
    Let $\B$ be a weakly $2$-presentably symmetric monoidal $2$-category, and let $b\in\B^\dbl$. For $i,j\in \mathbb N$, we let $\can_{i,j}: b^{\otimes i}\otimes (b^\vee)^{\otimes j}\to \Frrig_\B(b)$ denote the oplax cocone map from $\Cob_\B(b)(i^+\coprod j^-)\simeq b^{\otimes i}\otimes (b^\vee)^{\otimes j}$ to $\colim^\oplax_\Cob\Cob_\B(b) = \Frrig_\B(b)$. 
\end{nota}

We also need to describe a certain module structure. 
\begin{cons}\label{cons:module}
    Let $\B$ be a $2$-category and $f: a\to b$ be a map in $\B$ admitting a right adjoint. In this case, $f^R f\in\End_\B(a)$ is an algebra - in fact, it is the endomorphism algebra of $f\in\Hom_\B(a,b)$ under the action of $\End_\B(a)$ on the latter by precomposition. 

    Furthermore, $f^Rf$ is also the endomorphism ring of $f^R$ in $\Hom_\B(b,a)$ (where now the action of $\End_\B(a)$ is by postcomposition) and hence it acts on $f^R$, and therefore also on $f^Rg$ for each $g:c\to b$.   
\end{cons}
\begin{rmk}
    By the Eckmann-Hilton argument, if $\B$ is a symmetric monoidal $2$-category and $f$ is the unit map $\eta: \one_\B\to A$ if some commutative algebra, the algebra structure on $\eta^R\eta$ described above agrees with the one from \Cref{defn:End1}. 
\end{rmk}

\begin{cor}\label{cor:mapformula2}
    Let $\B$ be wealy $2$-presentably symmetric monoidal, and $b\in\B^\dbl$. Fix $i,j \in \N$ and consider the composite $b^{\otimes i}\otimes (b^\vee)^{\otimes j}\xrightarrow{\can_{i,j}} \Frrig_\B(b) \xrightarrow{\eta^R} \one_\B$. 

    As a module over $\End(\one_{\Frrig(b)})$, cf. \Cref{cons:module}, it is a coproduct indexed by the set of partitions of bijections $\sigma: i\to j$, where for each bijection $\sigma$, we take the free $\End(\one_{\Frrig(b)})$-module on $\ev_\sigma: b^{\otimes i}\otimes (b^\vee)^{\otimes j}\to \one_\B$, where $\ev_\sigma$ is the tensor product indexed by $i$ of the evaluation pairing $b\otimes b^\vee \to \one_\B$, pairing the $i$th copy of $b$ with the $\sigma(i)$th copy of $b^\vee$. 
\end{cor}
In particular, if $i\neq j$, there are no such bijections and the composite in question is the initial object in $\Hom_\B(b^{\otimes i}\otimes (b^\vee)^{\otimes j},\one_\B)$.
\begin{proof}
    That the underlying object is correct can be computed using the colimit formula from \Cref{cor:mapformula} and the analogous fact for mapping anima in the cobordism category $\Cob$.  

    To identify the module structure, we again look at the universal example $\Fun(\Cob\op,\PrL)$, where $\eta^R\circ \can_{i,j}$ is terminal in $\Hom_\B(b^{\otimes i}\otimes (b^\vee)^{\otimes j},\one_\B)$ by \Cref{cor:mapformula}, so that there, there is a unique module structure, thus proving the claim. 
\end{proof}
\begin{rmk}
    We stated this corollary for the specific pair of objects $i^+\coprod j^-$ and $\emptyset$ of $\Cob$, but of course there is a more general statement involving cobordisms $i_0^+\coprod j_0^-\to i_1^+\coprod j_1^-$. The above simply seems more manageable and the module structure is easier to define (though once it is defined, the description is as easy as in this case, for the same reason!). This is sufficient for most purposes. 
\end{rmk}
We now extract some concrete corollaries when $\B=\Mod_\V(\PrL)$ for some $\V\in\CAlg(\PrL)$. In this context, we let $\HH(-/\V)$ denote $\Dim$, cf. \cite[Section 4.5]{HSS}.

\begin{cor}\label{cor:maintextV}
    Let $\V\in\CAlg(\PrL)$ and $\M$ a dualizable $\V$-module. The free rigid commutative $\V$-algebra $\Frrig_\V(\M)$ on $\M$ exists, and has $\one_\V\{\HH(\M/\V)_{hS^1}\}$ as the endomorphism algebra of objects. 

    Furthermore, if $m_1,...,m_i\in \M$ are objects and $n_1,...,n_j\in\M$ are $\V$-atomic objects in the sense of \cite[Definition 1.22]{Dbl}, then we have, in $\V$, an equivalence of $\End(\one_{\Frrig_\V(\M)})$-modules: $$\End(\one_{\Frrig_\V(\M)})\otimes \bigoplus_{\sigma: i\cong j}\bigotimes_i\hom_\M(n_{\sigma(i)}, m_i)\simeq \hom_{\Frrig_\V(\M)}(\one, \bigotimes_i m_i\otimes \bigotimes_j n_j^\vee)$$
\end{cor}
\begin{proof}
For each $k$, let $N_k\in\M^\vee = \Fun^L_\V(\M,\V)$ denote $\hom_\M(n_k,-)$, which is well-defined since $n_k$ is atomic. 

Letting $\iota:\M\to \Frrig_\V(\M)$ denote the canonical inclusion and $\iota^\vee: \M^\vee\to \Frrig_\V(\M)$ also, we verify easily that $\iota^\vee(N_k)\simeq \iota(n_k)^\vee$ for all $k$. Hence $$\hom_{\Frrig_\V(\M)}(\one, \bigotimes_i m_i\otimes\bigotimes_j n_j^\vee)\simeq \hom_{\Frrig_\V(\M)}(\eta\one_\V, \can_{i,j}(\bigboxtimes_i m_i\boxtimes \bigboxtimes_j N_j))\simeq \eta^R\can_{i,j}(\bigboxtimes_i m_i\boxtimes \bigboxtimes_j N_j).$$
The description in terms of the left hand side now follows from \Cref{cor:mapformula2} and the fact that the pairing $\M\otimes\M^\vee\to \V$ is simply evaluation, and so sends $m_p\boxtimes N_q$ to $\hom_\M(n_q,m_p)$. 
\end{proof}

\begin{rmk}\label{rmk:moregenform}
    We extracted from our main theorem a concrete calculation of certain homs in $\Frrig_\V(\M)$, but the description as an oplax colimit is of course more precise. For example, this description is valid also when $\M$ is dualizable and has, potentially no nontrivial atomics. One way to see this in action is to plug in directly $N_k$'s in $\M^\vee$ as opposed to specifically those of the form $\hom_\M(n_k,-)$ for $n_k\in \M$ atomic. 
\end{rmk}

\begin{rmk}
    We point out, and this is relevant to \cite{KNS}, that \Cref{cor:End1} and \Cref{cor:mapformula} also show that in general, if $\M\to \mathcal{N}$ is fully faithful, this need not be the case for $\Frrig_\V(\M)\to \Frrig_\V(\mathcal{N})$. This is in contrast to, for example the case of free commutative algebras. The main obstruction is the endomorphism algebra of the unit, and as we explain below, this is the only obstruction.
\end{rmk}
\renewcommand{\N}{\mathcal{N}}
\begin{cor}\label{cor:ff}
Let $\V\in\CAlg(\PrL)$ and let $i:\M\to \N$ be a fully faithful internal left adjoint between dualizable $\V$-modules. 

The map $\Frrig_\V(\M)\otimes_{\End(\one_{\Frrig_\V(\M)})}\End(\one_{\Frrig_\V(\N)})\to \Frrig_\V(\N)$ is fully faithful. 

In particular, if furthermore $\HH(\M/\V)\to \HH(\N/\V)$ is an equivalence, then $$\Frrig_\V(\M)\to \Frrig_\V(\N)$$ is fully faithful. 
\end{cor}
Here, given a commutative $\V$-algebra $\mathcal{A}$ with unit $\one_\mathcal{A}$ and a map $\End(\one_\mathcal{A})\to R$ of commutative algebras in $\V$, we let $\mathcal{A}\otimes_{\End(\one_\mathcal{A})}R$ denote the relative tensor product of commutative $\V$-algebras $\mathcal{A}\otimes_{\Mod_{\End(\one_\mathcal{A})}(\V)}\Mod_R(\V)$. 
\begin{proof}
The ``In particular'' statement follows from the first statement using \Cref{cor:End1}.

Let $A:=\Frrig_\V(\M)\otimes_{\End(\one_{\Frrig_\V(\M)})}\End(\one_{\Frrig_\V(\N)})$ and $B:=\Frrig_\V(\N)$ and let $f:A\to B$ denote the relevant functor. 

 By \Cref{cor:rigff2}, it suffices to prove that $\eta_A^R\to \eta_A^Rf^R f$ is is an equivalence, i.e. 
 it suffices to prove that $\eta_A^R\to \eta_B^R f$ is an equivalence, i.e. that $\hom_A(\one_A,-)\to \hom_B(\one_B,f(-))$ is an equivalence. Both are $\V$-linear colimit-preserving functors $A\to \V$, and the $\can_{i,j}:\M^{\otimes i}\otimes (\M^\vee)^{\otimes j}\to \Frrig_\V(\M)\to A$ jointly generate $A$ as a $\V$-module, hence, it suffices to prove that $\hom_A(\one_A,\can_{i,j})\to \hom_B(\one_B,f\circ\can_{i,j})$ is an equivalence for all $i,j$. 

Since $\one_{\Frrig_\V(\M)}$ is atomic, the map $\hom_{\Frrig_\V(\M)}(\one,\can_{i,j})\to \hom_A(\one_A,\can_{i,j})$ witnesses the target as the extension of scalars along $\End(\one_{\Frrig_\V(\M)})\to \End(\one_{\Frrig_\V(\N)})$ of the source. 

Now, since for $F\in\M^\vee$, $m\in\M$, $i^\vee(F)(i(m))\simeq F(i^Ri(m))\simeq F(m)$, \Cref{cor:mapformula2} shows that the same is true for $\hom_{\Frrig_\V(\M)}(\one,\can_{i,j})\to \hom_{\Frrig_\V(\N)}(\one,f\circ \can_{i,j})$, so that the claim follows. 
\end{proof}
\begin{rmk}
    The ``In particular'' part of the corollary can also be shown in a general weakly $2$-presentably symmetric monoidal $2$-category $\B$, essentially using the last part of the proof as is. Fully faithfulness is going to guarantee that the oplax transformation $\Cob_\B(\M)\to \Cob_\B(\N)$ of functors $\Cob\to \B$ is strict along cobordisms of the form $+\coprod - \to \emptyset$, which suffices in conjunction with \Cref{cor:mapformula2}. 

    The whole corollary is probably also valid in this generality if one makes sense of the relative tensor product over endomorphisms of the unit. 
\end{rmk}
We provide a somewhat surprising application of this result, which answers a question asked to the author by Thomas Nikolaus:
\begin{cor}\label{cor:ffrigidemb}
    Let $C\in\CAlg^\rig(\PrL_\st)$. There exists a compactly generated rigid commutative algebra $D$ with a fully faithful symmetric monoidal embedding $C\to D$.
\end{cor}
\begin{rmk}
    It is a result of Lurie's that any dualizable presentable stable category embeds fully faithfully and in an internally left adjoint way into a compactly generated category. This is in some sense a multiplicative analogue of this result. 
\end{rmk}
\begin{proof}
    We let $\Frrig_\st:=\Frrig_\Sp$. 

    By \cite[Example 4.14]{Dbl}, there exists a compactly generated $D_0$ and a fully faithful embedding $C\to D_0$ which is a motivic equivalence in the sense of \cite{RSW}, and in particular induces an equivalence on $\THH := \HH(-/\Sp)$.
    
   We now consider the following relative tensor product: $D:=C\otimes_{\Frrig_\st(C)}\Frrig_\st(D_0)$. There is clearly a map $C\to D$ in $\CAlg^\rig(\PrL_\st)$. Furthermore, \Cref{cor:ff} shows that $\Frrig_\st(C)\to\Frrig_\st(D_0)$ is fully faithful, hence so is $C\to D$. 

    Finally, $\Frrig_\st(C)\to C$ is split surjective so that the image of $\Frrig_\st(D_0)$ in $D$ generates $D$ under colimits, and $\Frrig_\st(D_0)$ is compactly generated, therefore so is $D$ (note that the map $\Frrig_\st(D_0)\to D$ is a map between rigid commutative algebras and hence is an internal left adjoint, and thus preserves compacts).
\end{proof}
\begin{ques}
    Let $C=\mathrm{Sh}(X,\Sp)$ for $X$ a compact Hausdorff topological space. Is there a natural/geometric candidate of a $D$ as in the corollary ?
\end{ques}
Another consequence we find is that any dualizable category embeds fully faithfully in a rigid one, at least in the pointed setting:
\begin{cor}
    Let $\V\in\CAlg(\PrL_*)$, and $\M$ a dualizable $\V$-module. There exists a rigid commutative $\V$-algebra $\W$ and a fully faithful $\V$-internal left adjoint $\M\to \W$.  
\end{cor}
\begin{proof}
    Consider the map $\End(\one_{\Frrig_\V(\M)})\to \one_\V$ induced by the zero map $\HH(\M/\V)_{hS^1}\to \one_\V$ using \Cref{cor:maintextV}. From \Cref{cor:maintextV} (and \Cref{rmk:moregenform}), we find that the composite $$\M\to\Frrig_\V(\M)\to \Frrig_\V(\M)\otimes_{\End(\one_{\Frrig_\V(\M)})}\one_\V$$ is fully faithful. 
\end{proof}

\begin{rmk}
    Another consequence of our main result is the ``motivic invariance'' of $\Frrig_\V$, in the sense that it sends $\V$-motivic equivalences (in the sense of \cite[Appendix B]{RSW}) to $\V$-motivic equivalences. We shall not prove this here, but return to it in future work.
\end{rmk}
\section{Consequences for locally rigid algebras}\label{section:locrig}
A consequence of our main theorem is the adjointability property we state below. The goal of this section is to extract consequences of this adjointability property for the theory of locally rigid commutative algebras in the sense of \cite{LocRig} -- consequences which we either already proved with difficult arguments, or which we were unable to prove in this generality. 
\begin{cor}\label{cor:adjtable}
    Let $f:\B_0\to \B_1$ be a $2$-colimit-preserving, locally colimit-preserving symmetric monoidal $2$-functor between weakly $2$-presentably symmetric monoidal $2$-categories. 

    The following commutative diagram is vertically left adjointable:
    \[\begin{tikzcd}
	{\CAlg^\rig(\B_0)} & {\CAlg^\rig(\B_1)} \\
	{\B_0^\dbl} & {\B_1^\dbl}
	\arrow[from=1-1, to=1-2]
	\arrow[from=1-1, to=2-1]
	\arrow[from=1-2, to=2-2]
	\arrow[from=2-1, to=2-2]
\end{tikzcd}\]
\end{cor}
\begin{proof}
    This follows from the fact that $f$ commutes with oplax colimits, and the naturality in $\B$ of the retraction from \Cref{cor:exist}. 
\end{proof}

That this Corollary is not obvious was explained in \Cref{rmk:adjt}. We now intend to derive consequences from it in the context\footnote{They can be drawn more generally, but the author does not yet know the optimal generality, and we will therefore restrict to this setting for now.} of $\B=\Mod_\V(\PrL)$. The first one is the following version of \cite[Theorem C]{LocRig}, which was proved with very different means in \textit{loc. cit.}:
\begin{cor}
    Let $\V\in\CAlg(\PrL)$ and let $\W$ be a locally rigid commutative $\V$-algebra in the sense of \cite{LocRig}. There exists a rigid commutative $\V$-algebra $\overline{\W}$ and a commutative $\V$-algebra map $\overline{\W}\to \W$ admitting a fully faithful $\V$-linear left adjoint (and hence $\overline{\W}$-linear as well). 

    In fact, one can choose $\overline{\W}$ to be the $\V$-rigidification of $\W$. 
\end{cor}
\begin{proof}
    Note that the basechange functor $\W\otimes_\V-$ induces a functor $\Mod_\V^\dbl\to \Mod_\W^\dbl$ and similarly, $\CAlg^\rig_\V\to \CAlg^\rig_\W$. Since all of these categories are presentable and the functors preserve colimits, they all have right adjoints. By right adjointing the adjointable square from \Cref{cor:adjtable}, we find that the following square is horizontally right adjointable:
    \[\begin{tikzcd}
	{\CAlg^\rig_\V} & {\CAlg^\rig_\W} \\
	{\Mod_\V^\dbl} & {\Mod_\W^\dbl}
	\arrow[from=1-1, to=1-2]
	\arrow[from=1-1, to=2-1]
	\arrow[from=1-2, to=2-2]
	\arrow[from=2-1, to=2-2]
\end{tikzcd}\]
We call $R$ the right adjoint to $\W\otimes_\V -$, and abusively the same letter for the right adjoint at the level of rigid commutative algebras - since the square is adjointable, this abuse should not cause confusion.

The basechange functor is symmetric monoidal and so the bottom adjunction factors as $$\Mod_\V^\dbl\rightleftharpoons \Mod_{R(\W)}(\Mod_\V^\dbl)\rightleftharpoons \Mod_\W^\dbl$$

Since $R(\W)$ is rigid, by adjointability, we find that $\Mod_{R(\W)}(\Mod_\V^\dbl)\simeq \Mod_{R(\W)}^\dbl$ by \cite[Corollary 4.46]{LocRig}, and unwinding the construction, we find that $\W$ is a commutative $R(\W)$-algebra and that the induced adjunction is the same as the original one where $R(\W)$ plays the role of $\V$. Since $R(\W)$ is rigid over $\V$ and $\W$ is locally rigid over $\V$, $\W$ is locally rigid over $R(\W)$ \cite[Corollary 4.19]{LocRig}. All in all, replacing $\V$ with $R(\W)$, we may assume without loss of generality that $R(\W) = \V$. 

Now, since $\W$ is dualizable over $\V$, there is an additional forgetful functor $\iota: \Mod_\W^\dbl\to \Mod_\V^\dbl$. 
\begin{warn}
    If $\W$ is not rigid, $\iota$ is \emph{not} right adjoint to $\W\otimes_\V-$ at the level of categories of dualizable modules! Of course, they are right adjoint when considering all modules and all maps.  
\end{warn}
For any $\W$-module $\M$, the multiplication map $\W\otimes_\V\iota(\M) \to \M$ is an internal left adjoint by \cite[Proposition 4.11]{LocRig}. Hence, it induces by adjunction a map $s:\iota(\M)\to R(\M)$ for all $\W$-modules $\M$. We claim that this map is fully faithful. Granting this, we find that $\iota(\W)\to R(\W)=\V$ is fully faithful, and since $\W$ has a $\V$-algebra structure, it witnesses $\iota(\W)$ has a coidempotent ideal in $\V$. 

Coidempotent ideals $\mathcal I\to \V$ admit a unique commutative $\V$-algebra structure such that the right adjoint $\V\to\mathcal I$ is symmetric monoidal, and this commutative algebra structure is determined by its multiplication. 

Since the inclusion map $i:\W\to R(\W)$ has, by design, the property that $\W\otimes i: \W\otimes\W\to \W\otimes R(\W)\to \W$ is the original multiplication map of $\W$, we deduce that this unique commutative algebra structure is the one we started with. In other words, the inclusion $i$ is left adjoint to the unit map $\V=R(\W)\to \W$, which proves the first part of the claim. 

Thus, we are left with proving that $s$ is indeed fully faithful. This follows by examining universal properties. Indeed, for any $\V$-module $\N$, we have the following diagram:
\[\begin{tikzcd}
	{\Fun^{iL}_\V(\N,\iota(\M))} & {\Fun^L_\V(\N,\iota(\M))} \\
	{\Fun^{iL}_\W(\W\otimes_\V\N,\M)} & {\Fun^L_\W(\W\otimes_\V\N,\M)} \\
	{\Fun^{iL}_\V(\N,R(\M))}
	\arrow[hook, from=1-1, to=1-2]
	\arrow["\simeq", from=1-2, to=2-2]
	\arrow[hook, from=2-1, to=2-2]
	\arrow["\simeq"', from=2-1, to=3-1]
\end{tikzcd}\]
where the right vertical arrow is given by the adjunction $\W\otimes_\V-\dashv \iota$ that exists on $\Mod_\V(\PrL)$ and $\Mod_\W(\PrL)$, and the left vertical arrow is given by the definition of $R$. We now claim that the right vertical arrow preserves the full subcategories spanned by internal left adjoints: indeed, unwinding the adjunction, it is given by sending a functor $f:\N\to \iota(\M)$ to $\W\otimes_\V\N\to \W\otimes_\V\iota(\M)\to \M$, and if $f$ is an internal left adjoint, so is this composite by \cite[Proposition 4.11]{LocRig}. In particular, the induced functor $\Fun^{iL}_\V(\N,\iota(\M))\to \Fun^{iL}_\V(\N,R(\M))$ is fully faithful. Unwinding definitions, this functor is given by postcomposition by $s$, and it follows that $s$ is fully faithful, which was to be shown. 

We finally claim that $R(\W)$ is the rigidification of $\W$. This follows immediately from the adjunction $\W\otimes_\V-: \CAlg^\rig_\V\rightleftharpoons \CAlg^\rig_\W: R$
\end{proof}

We conclude this section and the paper with a proof of the following result, which Efimov observed in the case $\V=\Sp$ in \cite[Proposition 4.1]{efimovII} - his proof involves an explicit analysis of a certain pre-compact ideal in the sense of \cite[Definition 2.7.1]{KNP}, and seems hard or at least tedious to reproduce in the setting of a general $\V$. Instead, we give a direct proof with no such analysis needed:
\begin{cor}
    Let $\V\in\CAlg(\PrL)$, and let $\W$ be a commutative $\V$-algebra such that the unit map $\V\to\W$ admits a fully faithful, $\V$-linear left adjoint. In this case, the $\V$-rigidification of $\W$ is given by the dualizable internal hom $\Hom^\dbl_\V(\W,\W)$. 
\end{cor}
\begin{rmk}
    Note that if $\V$ is rigid over $\V_0$, the $\V$- and $\V_0$-rigidifications of $\W$ agree \cite[Observation 4.66]{LocRig}. Thus, this provides the following ``recipe'' to describe rigidifications: for a locally rigid $\W$ over $\V$, find a rigid $\overline{\W}$ such that $\W$ is an internally left adjointly embedded coidempotent ideal in $\overline{\W}$, and describe $\Hom^\dbl_{\overline{\W}}(\W,\W)$. 
\end{rmk}
\begin{proof}
    In this situation, we claim that the right adjoint $R$ to $\W\otimes_\V-: \Mod_\V^\dbl\to \Mod_\W^\dbl$ is given by $\Hom^\dbl_\V(\W,-)$. This is the case by the following calculation: 
    $$\Map^{iL}_\V(\M,\Hom^\dbl_\V(\W,-))\simeq \Map^{iL}_\V(\W\otimes_\V \M,-)\simeq \Map^{iL}_\W(\W\otimes_\V \M, -)$$
    where the last step uses that the source and target are $\W$-modules, and that for two $\W$-modules $\Fun^L_\V\simeq \Fun^L_\W$ since $\W$ is idempotent over $\V$, and finally that this equivalence respects the notions of internal left adjoints by \cite[Proposition 4.18]{LocRig}. 

    Now, as in the proof of the previous corollary, note that $R(\W)$ is the rigidification of $\W$ as a direct consequence of \Cref{cor:adjtable}. 
\end{proof}
\begin{rmk}
    Note that in the situation of the corollary, the forgetful functor $\iota$ is actually \emph{left} adjoint to the basechange functor $\W\otimes_\V-$, since $\W$ can be viewed as a coidempotent coalgebra over $\V$. Since $\iota$ is fully faithful, it follows that $R=\Hom^\dbl_\V(\W,-)$ is as well, and hence, for every dualizable $\W$-module $\M$, $\W\otimes_\V\Hom^\dbl_\W(\W,\M)\simeq \M$. 
\end{rmk}

\appendix
\section{Oplax colimits of lax symmetric monoidal functors}\label{app:construction}
The goal of this appendix is to give the details behind \Cref{cons:colim}. We give a somewhat roundabout construction/proof, not because we believe it is the best or the most reasonable, but because we lack the necessary $2$-categorical technology to give the ``right'' conceptual proofs. 

Our strategy is to construct a functor $C:\Fun^{\otimes-\lax}_\oplax(I,\B)\to \CAlg(\B)$ making the following square commute: \[\begin{tikzcd}
	{\Fun^{\otimes-\lax}_\oplax(I,\B)} & {\CAlg(\B)} \\
	{\Fun_\oplax(I,\B)} & \B,
	\arrow[from=1-1, to=1-2]
	\arrow[from=1-1, to=2-1]
	\arrow[from=1-2, to=2-2]
	\arrow["{\colim^\oplax_I}"', from=2-1, to=2-2]
\end{tikzcd}\] 
together with a unit transformation $\id\to \Delta C$ lifting the obvious unit transformation, and to verify that this participates in an adjunction. Our construction of $C$ (and the unit transformation) is rather roundabout: we construct it by hand in the universal example, which gives it in general at the level of groupoid cores, and we then extend it in general using the Yoneda lemma and \Cref{recoll:Funoplax}.

We begin with the universal example, for which we shall need a way of constructing symmetric monoidal oplax natural transformations. We do this using the following microscopic version of monoidal straightening:

\begin{prop}\label{prop:laxmongro}
    Let $\mathcal E_0,\mathcal E_1\to \mathcal B$ be symmetric monoidal coCartesian fibrations in the sense of \cite[Definition 1.11]{MonGro}, and $f:\mathcal E_0\to\mathcal E_1$ a lax symmetric monoidal functor over $\mathcal B$. 

    The corresponding straightened lax natural transformation between the straightened functors $\mathcal B\to \Cat$ as in \cite[Theorem A]{abellan2024straightening} can be upgraded to a symmetric monoidal lax natural transformation, i.e. a map in $\Fun_\lax^{\lax-\otimes}(\mathcal B,\Cat)$. 
\end{prop}
To prove this, we will use the following lemmas:
\begin{lm}
    Let $C$ be a $2$-category with $2$-categorical finite products. For any small category $K$, $\Fun_{(\mathrm{op})\lax}(K,C)$ has finite products, which are preserved by all evaluation functors to $C$. 
\end{lm}
\begin{proof}
Since products are (op)lax limits, this is a special case of (the dual of) \Cref{lm:oplaxcolimsinlax}
\end{proof}
\begin{lm}
    Let $C$ be a $2$-category with $2$-categorical products, considered as a cartesian symmetric monoidal $2$-category, and $I$ a symmetric monoidal category. There is an equivalence between the full subcategory of $\Fun_\lax(I^\otimes,C)$ spanned by lax cartesian structures in the sense of \cite[Definition 2.4.1.1]{HA} and $\Fun_\lax^{\lax-\otimes}(I,C)$, which is given on objects by the equivalence from \cite[Proposition 2.4.1.7]{HA}. 
\end{lm}
\begin{proof}
   We have equivalences natural in the category $K$: $$\Map(K,\Fun_\lax^{\lax-\otimes}(I,C))\simeq \Map^{\lax-\otimes}(I, \Fun_\oplax(K,C)) \simeq \Map^{\lax-\cart}(I^\otimes, \Fun_\oplax(K,C))$$ 
   $$\simeq \Map(K,\Fun^{\lax-\cart}_\lax(I^\otimes,C)) $$
   where for $D$ a category with finite products, $\Map^{\lax-\cart}(I^\otimes,D)$ is the full sub-anima of $\Map(I^\otimes,
   D)$ spanned by lax cartesian structures in the sense of \cite[Definition 2.4.1.1]{HA}. 

   Here, the first equivalence is by definition (see \Cref{recoll:Funoplax}), the second uses \cite[Proposition 2.4.1.7]{HA} in conjunction with the previous lemma, and the third one follows from the fact that products in $\Fun_\oplax(K,\Cat)$ are pointwise (and the usual interaction of lax and oplax functor categories). 
\end{proof}
\begin{proof}[Proof of \Cref{prop:laxmongro}]
We note that $f$ corresponds to a functor $\mathcal E_0^\otimes\to \mathcal E_1^\otimes$ between coCartesian fibrations over $\mathcal B^\otimes$ (though it need not preserve cocartesian edges). By lax straightening \cite[Theorem A]{abellan2024straightening}, it induces a lax natural transformation of functors $\mathcal B^\otimes \to \Cat$. Both functors are lax cartesian structures, so the previous lemma implies that this corresponds to a single lax natural transformation of lax symmetric monoidal functors $\mathcal B\to \Cat$.
\end{proof}

Now, the universal example will look as follows: given a category $J$, the Yoneda embedding $y_J: J\to \Fun(J\op,\Cat)$ has the following oplax colimit: it is the presheaf $$Y_J: j\mapsto (J_{j/})\op,$$ with coCartesian unstraightening given by the source functor $s\op:\Ar(J)\op\to J\op$. This gives a functor $J\to \Fun_\lax(\Delta^1,\Fun(J\op,\Cat))$ classifying the oplax transformation $y_J\to \Delta Y_J$. 

Using the un/straightening equivalence, we can instead view this as a functor $P_J: J\to \Fun_\lax(\Delta^1,\coCart_{J\op})$. In symbols, this functor sends $j$ to the map $(J_{/j})\op\to \Ar(J)\op$ of coCartesian fibrations over $J\op$. 
\begin{warn}
    We warn the reader to be careful to what objects are what: here, $(J_{/j})\op$ and $(J_{j/})\op$ both feature but play different roles. The former is the image of the Yoneda embedding at $j$, i.e. $y_J(j)$, viewed as a coCartesian fibration over $J\op$. In particular, it is a the image at $j$ of a functor $J\to \coCart_{J\op}$.  

    The latter, on the other hand, is the value at $j$ of the oplax colimit of the former, so it is the fiber at $j$ of a single object in $\coCart_{J\op}$, namely $\Ar(J)\op$. It is better to keep the latter in mind. 
\end{warn}

We will make $P_J$ a symmetric monoidal lax natural transformation, so that when we plug in $J=I$ a symmetric monoidal category, it will give us a lax symmetric monoidal functor $P_I: I\to \Fun_\lax(\Delta^1,\Fun(I\op,\Cat))$, i.e. a symmetric monoidal oplax natural transformation $y_I\to \Delta Y_I$ (and in particular, it will give $Y_I$ a commutative algebra structure in $\Fun(I\op,\Cat)$).

To make $P_J$ a symmetric monoidal lax transformation, we will use \Cref{prop:laxmongro}. First, we note that $P_J$ is a lax natural transformation (with no monoidality properties at this point):
\begin{prop}\label{prop:laxnat}
    The functor $P_J$ is the fiber over $J\in \Cat$ of a single functor $$P:\Cat_{*//}\to \Fun_\lax(\Delta^1,\coCart)\times_{\Fun_\lax(\Delta^1,\Cat)}\Cat$$ over $\Cat$.

    The functor $P$ preserves finite products. 
\end{prop}
\begin{proof}
The map $\Cat_{\Delta^1//}\to \Cat_{*//}$ given by precomposition with the inclusion $*\xrightarrow{1}\Delta^1$ is a coCartesian fibration classifying $(J,j)\mapsto J_{/j}$ while the pullback of the composite $\Cat_{\Delta^1//}\to \Cat_{*//}\to \Cat$ along $\Cat_{*//}\to \Cat$ classifies $(J,j)\mapsto \Ar(J)$. 

The canonical map $\Cat_{\Delta^1//}\to \Cat_{\Delta^1//}\times_\Cat \Cat_{*//}$ is a map between coCartesian fibrations over $\Cat_{*//}$ which classifies the lax natural transformation $J_{/j}\to \Ar(J)$ (lax natural in the $(J,j)$ variable) which, for fixed $J$, is a lax colimit lax cocone of the variable $j\in J$. 

We then have the pullback $\Cat_{*//}\times_\Cat \Cat_{*//}$ which classifies $(J,j)\mapsto J$ and the natural map $\Cat_{\Delta^1//}\to \Cat_{*//}$ induces a map $\Cat_{\Delta^1//}\times_{\Cat}\Cat_{*//}\to \Cat_{*//}\times_\Cat \Cat_{*//}$ which classifies the source projection $\Ar(J)\to J$. 

The commutative square 
\[\begin{tikzcd}
	{\Cat_{\Delta^1//}} & {\Cat_{\Delta^1//}\times_{\Cat}\Cat_{*//}} \\
	{\Cat_{*//}\times_\Cat \Cat_{*//}} & {\Cat_{*//}\times_\Cat \Cat_{*//}}
	\arrow[from=1-1, to=1-2]
	\arrow[from=1-1, to=2-1]
	\arrow[from=1-2, to=2-2]
	\arrow[shift right, no head, from=2-1, to=2-2]
	\arrow[shift left, no head, from=2-1, to=2-2]
\end{tikzcd}\]
of maps over $\Cat_{*//}$ between fibrations straightens under \cite[Theorem A]{abellan2024straightening} to the desired functor $\Cat_{*//}\to \Fun_\lax(\Delta^1,\coCart)\times_{\Fun_\lax(\Delta^1,\Cat)}\Cat$, up to an $(-)\op$. More precisely, it lands in $\Fun_{\oplax}(\Delta^1\times\Delta^1,\Cat)\times_{\Fun_\oplax(\Delta^1,\Cat)}\Cat$, but one then simply \emph{checks} that it lands in the appropriate subcategory of cartesian fibrations. We can then apply the $(-)\op$ equivalence between $\cart^{\mathrm{co}}$ and $\coCart$ (which therefore turns the $\oplax$ into a $\lax$) to get the desired functor.
\end{proof}

\begin{cor}
    The lax natural transformation $P_J: J\to \Fun_\lax(\Delta^1,\coCart_{J\op})$ refines to a symmetric monoidal lax natural transformation. In particular for $J=I$ a symmetric monoidal category, it induces a lax symmetric monoidal functor $P_I: \Fun_\lax(\Delta^1,\coCart_{I\op})\simeq \Fun_\lax(\Delta^1,\Fun(I\op,\Cat))$, or equivalently, a symmetric monoidal oplax transformation of functors $y_I\to \Delta Y_I$. 
\end{cor}
\begin{proof}
    The functor $P:\Cat_{*//}\to \Fun_\lax(\Delta^1,\coCart)\times_{\Fun_\lax(\Delta^1,\Cat)}\Cat$ is over $\Cat$ and preserves finite products, and it therefore (uniquely) induces a symmetric monoidal functor between the cartesian symmetric monoidal structures on these categories, over $\Cat^\times$. The result thus follows from \Cref{prop:laxmongro}. 
\end{proof}
\begin{obs}\label{obs:monstr}
    Unwinding the construction, we see that the symmetric monoidal oplax transformation $y_I\to \Delta Y_I$ is \emph{monoidally strict}, in the sense that the oplax commutative squares
    \[\begin{tikzcd}
	{y_I(i)\otimes y_I(j)} & {Y_I\otimes Y_I} \\
	{y_I(i\otimes j)} & {Y_I}
	\arrow[from=1-1, to=1-2]
	\arrow[from=1-1, to=2-1]
	\arrow[from=1-2, to=2-2]
	\arrow[from=2-1, to=2-2]
\end{tikzcd}\] 
commute strictly. 
\end{obs}
\begin{obs}
    By construction, the underlying oplax cocone $y_I\to \Delta Y_I$ is an oplax colimit cocone. 
\end{obs}
This deals with the universal example. We deduce the following: 
\begin{cor}\label{cor:consgpdcoco}
Fix $I$ a symmetric monoidal category. 
    There is a natural transformation in the $2$-cocompletely symmetric monoidal $2$-category $\B$ $$\Map^{\otimes-\lax}(I,\B)\to \CAlg(\B)^\simeq \times_{\Map^{\otimes-\lax}(I,\B)} \Map(\Delta^1,\Fun^{\otimes-\lax}_\oplax(I,\B)) $$ fitting into a commutative square: 
    \[\begin{tikzcd}
	{\Map^{\otimes-\lax}(I,\B)} & {\CAlg(\B)^\simeq \times_{\Map^{\otimes-\lax}(I,\B)} \Map(\Delta^1,\Fun^{\otimes-\lax}_\oplax(I,\B))} \\
	{\Map(I,\B)} & {\B^\simeq\times_{\Map(I,\B)}\Map(\Delta^1,\Fun_\oplax(I,\B))}
	\arrow[from=1-1, to=1-2]
	\arrow[from=1-1, to=2-1]
	\arrow[from=1-2, to=2-2]
	\arrow["{\colim^\oplax_I}"', from=2-1, to=2-2]
\end{tikzcd}\]
where the bottom horizontal arrow is the oplax colimit functor together with the associated unit map. 
\end{cor}
\begin{proof}
    By the Yoneda lemma, it suffices to produce an element in $$\CAlg(\Fun(I\op,\Cat))\times_{\Map^{\otimes-\lax}(I,\Fun(I\op,\Cat))} \Map(\Delta^1,\Fun^{\otimes-\lax}_\oplax(I,\Fun(I\op,\Cat)))$$ lifting the oplax colimit over $I$ of the Yoneda embedding and the unit map. 

    We did this in the previous corollary. 
\end{proof}

It is conceptually better to generalize this to symmetric monoidal $2$-categories $\B$ that only admit oplax colimits indexed by powers of $I$, but it will also be helpful for our construction. We therefore point out that this actually follows from the above corollary by embedding our categories symmetric monoidally and oplax-colimit-over-powers-of-$I$-preservingly into $2$-cocompletely symmetric monoidal $2$-categories:
 \begin{cor}\label{cor:transgpd}
     Fix $I$ a symmetric monoidal category. 
    The natural transformation 
    
    $$\Map^{\otimes-\lax}(I,\B)\to \CAlg(\B)^\simeq \times_{\Map^{\otimes-\lax}(I,\B)} \Map(\Delta^1,\Fun^{\otimes-\lax}_\oplax(I,\B)) $$ from \Cref{cor:consgpdcoco}, as well as the commutative square following it, extend to the category of symmetric monoidal $2$-categories with oplax colimits indexed by finite powers of $I$ compatible with the tensor product. 
    \end{cor}

    We now extend this to a functor on $\Fun^{\otimes-\lax}_\oplax(I,\B)$ rather than on its groupoid core by using the following observation, which was obtained during a conversation with Thorger Gei{\ss}: 
    \begin{lm}\label{lm:oplaxcolimsinlax}
    Let $\B$ be a $2$-category admitting $J$-indexed oplax colimits and $K$ a small $2$-category. In this case, $\Fun_\lax(K,\B)$ admits $J$-indexed oplax colimits, and they are underlying in the sense that evaluation functors $\Fun_\lax(K,\B)\to \B$ preserve them. 
\end{lm}
\begin{proof}
    Apply the $2$-functor $\Fun_\lax(K,-)$ to the oplax colimit/diagonal adjunction $\Fun_\oplax(J,\B)\rightleftharpoons \B $, and use $\Fun_\lax(K,\Fun_\oplax(J,\B))\simeq \Fun_\oplax(J,\Fun_\lax(K,\B))$. 
\end{proof}

In particular, we obtain: 
\begin{cor}
    Let $\B$ be a symmetric monoidal $2$-category admitting oplax colimits indexed by finite powers of $J$ compatible with the tensor product. In this case, for any small category $K$, the $2$-category $\Fun_\lax(K,\B)$ is also a symmetric monoidal $2$-category admitting oplax colimits indexed by finite powers of $J$ compatible with the tensor product, and these oplax colimits are preserved by precomposition with any functor $K_0\to K_1$. 
\end{cor}
We can now produce the desired functor equipped with a lift of the unit: 
\begin{cor}\label{cor:consC}
    Fix $I$ a symmetric monoidal category. 
    There is a natural transformation in the $2$-cocompletely symmetric monoidal $2$-category $\B$ $$\Fun_\oplax^{\otimes-\lax}(I,\B)\to \CAlg(\B) \times_{\Fun_\oplax^{\otimes-\lax}(I,\B)} \Fun(\Delta^1,\Fun^{\otimes-\lax}_\oplax(I,\B)) $$ fitting into a commutative square: 
    \[\begin{tikzcd}
	{\Fun_\oplax^{\otimes-\lax}(I,\B)} & {\CAlg(\B)\times_{\Fun_\oplax^{\otimes-\lax}(I,\B)} \Fun(\Delta^1,\Fun^{\otimes-\lax}_\oplax(I,\B))} \\
	{\Fun_\oplax(I,\B)} & {\B\times_{\Fun_\oplax(I,\B)}\Fun(\Delta^1,\Fun_\oplax(I,\B))}
	\arrow[from=1-1, to=1-2]
	\arrow[from=1-1, to=2-1]
	\arrow[from=1-2, to=2-2]
	\arrow["{\colim^\oplax_I}"', from=2-1, to=2-2]
\end{tikzcd}\]
where the bottom horizontal arrow is the oplax colimit functor together with the associated unit map. 
\end{cor}
\begin{proof}
    Using \Cref{lm:oplaxcolimsinlax}, we can apply \Cref{cor:transgpd} to $\Fun_\lax(K,\B)$, functorially in $(K,\B)$. 

    Using $\Map^{\otimes-\lax}(I,\Fun_\lax(K,\B))\simeq \Map(K,\Fun^{\otimes-\lax}_\oplax(I,\B))$, we obtain by the Yoneda lemma an extension of the transformation on groupoids to a transformation of the form 
    $$\Fun_\oplax^{\otimes-\lax}(I,\B)\to \CAlg(\B)_\oplax \times_{\Fun_\oplax^{\otimes-\lax}(I,\B)} \Fun_\oplax(\Delta^1,\Fun^{\otimes-\lax}_\oplax(I,\B))$$

    Unwinding the definitions, we find that for a map $F\to G$ in $\Fun^{\otimes-\lax}_\oplax(I,\B)$, the induced oplax morphism of commutative algebras has oplax monoidality squares given by $\colim^\oplax_{I\times I}$ applied to the following oplax commuting square:
    \[\begin{tikzcd}
	{F\boxtimes F} & {G\boxtimes G} \\
	{F\circ m_I} & {G\circ m_I}
	\arrow[from=1-1, to=1-2]
	\arrow[from=1-1, to=2-1]
	\arrow[from=1-2, to=2-2]
	\arrow[from={0.2}{0.8}, Rightarrow, from=2-1, to=1-2]
	\arrow[from=2-1, to=2-2]
\end{tikzcd}\] 
composed with the naturality square for the canonical map $\colim^\oplax_{I\times I} (F\circ m_I)\to \colim^\oplax_I F$. By design, both of these squares are strictly commutative, so it follows that our functor actually lands in $\CAlg(\B)\subset \CAlg(\B)_\oplax$. 

Furthermore, since $\Fun^{\otimes-\lax}_\oplax(I,\B)\to \Fun_\oplax(I,\B)$ is locally conservative, i.e. conservative on hom categories, it follows that the $\Fun_\lax(\Delta^1,\Fun_\oplax^{\otimes-\lax}(I,\B))$ component actually lands in $\Fun(\Delta^1,\Fun^{\otimes-\lax}_\oplax(I,\B))$. 
\end{proof}

More informally, we have a functor $C:\Fun^{\otimes-\lax}_\oplax(I,\B)\to \CAlg(\B)$ which is given by a certain commutative algebra structure on $\colim^\oplax_I F$, for $F:I\to \B$ lax symmetric monoidal, which is such that the unit map $F\to \Delta \colim^\oplax_I F$ is compatible with the monoidal structures. 

We now describe locally a counit map:
\begin{cons}
    Let $A\in \CAlg(\B)$, where $\B$ is a $2$-cocompletely symmetric monoidal $2$-category.

$A$ is classified by a lax symmetric monoidal functor $\pt\to \B$, and by definition $\Delta(A)$ is the composite lax symmetric monoidal functor $I\xrightarrow{p} \pt \to \B$. Thus, the Kan extended functor $\Fun(I\op,\Cat)\to \B$ classifying $\Delta(A)$ factors as $\Fun(I\op,\Cat)\xrightarrow{p_\#}\Cat \to \B$, where $\Cat\to \B$ is the lax symmetric monoidal functor $K\mapsto K\otimes A$. Since $\pt$ is terminal in $\Cat$, there is a unique commutative algebra map $p_\# Y_I\to \pt$, where $Y_I= \colim^\oplax_I y_I$ is the oplax colimit of the Yoneda embedding equipped with its commutative algebra structure from \Cref{cor:consC}.

Thus we obtain a map $\colim^\oplax_I \Delta(A)\to A$ of commutative algebras in $\B$. Upon inspection, the underlying map of this counit is the counit of the ordinary oplax colimit adjunction.
\end{cons}
\begin{cor}
Let $I$ be a symmetric monoidal category, and let $\B$ be a symmetric monoidal $2$-category with oplax colimits indexed by finite tensor powers of $I$ compatible with tensor products. 

    The functor $C:\Fun^{\otimes-\lax}_\oplax(I,\B)\to \CAlg(\B)$ constructed in \Cref{cor:consC} is left adjoint to the diagonal functor. 
\end{cor}
\begin{proof}
    Without loss of generality, $\B$ is $2$-cocompletely symmetric monoidal. 

    In that case, recall that we already constructed a unit $F\to \Delta\circ C(F)$ natural in $F$, as well as, for each object $A\in\CAlg(\B)$, a counit $C\Delta(A)\to A$. 

    We are left with verifying that the following composites are equivalences: $C(F)\to C(\Delta\circ C(F))\to C(F)$ and $\Delta(A)\to \Delta\circ C\circ \Delta(A)\to \Delta(A)$. 

   The underlying maps of the unit and the counit are the unit and the counit of the ordinary oplax colimit adjunction applied to the underlying functor of $F$, and the underlying object of $A$. 

    Since $\CAlg(\B)\to \B$ is conservative, this proves that the first map is an equivalence. 

    For the second one, we note that the counit is a map of algebras, so $\Delta$ applied to it yields a strict natural transformation, and that in general, $F\to \Delta C(F)$ is monoidally strict (in the sense of \Cref{obs:monstr} -- though of course, it is not a strict transformation). In total the composite $\Delta(A)\to \Delta C\Delta(A)\to \Delta(A)$ is a monoidally strict transformation. The forgetful functor $\Fun^{\otimes-\lax}_\oplax(I,\B)\to \Fun_\oplax(I,\B)$ is conservative on monoidally strict transformations, and so it follows from the above discussion that this composite is an equivalence. 

    In total, we have verified both triangle identities, which shows the adjunction. 
\end{proof}

The fact that the unit map $F\to \Delta C(F)$ is a lift of the unit map of the ordinary oplax colimit adjunction furthermore shows the following, which is the precise statement behind \Cref{cons:colim}:
\begin{cor}
    Let $I$ be a symmetric monoidal category and $\B$ a symmetric monoidal $2$-category admitting oplax colimits indexed by finite tensor powers of $I$, compatible with the tensor product. 

    In this case, the diagonal functor $\CAlg(\B)\to\Fun_\oplax^{\otimes-\lax}(I,\B)$ admits a left adjoint and the following square is horizontally left adjointable:
\[\begin{tikzcd}
	{\Fun^{\lax-\otimes}_\oplax(I,\B)} & {\CAlg(\B)} \\
	{\Fun_\oplax(I,\B)} & \B
	\arrow[from=1-1, to=2-1]
	\arrow[from=1-2, to=1-1]
	\arrow[from=1-2, to=2-2]
	\arrow[from=2-2, to=2-1]
\end{tikzcd}\]
\end{cor}
\bibliographystyle{alpha}
\bibliography{Biblio.bib}

\end{document}